\documentclass[12pt]{amsart}
\include{xy}
\usepackage{mathrsfs}
\usepackage{fullpage}
\xyoption{all}

\newcommand{\Q}{{\mathbb Q}}

\newcommand{\CC}{{\mathscr C}}
\newcommand{\LL}{{\mathscr  L}}
\newcommand{\FF}{{\mathscr F}}
\newcommand{\pr}{\operatorname{pr}}

\newcommand{\EE}{{\mathscr E}}
\newcommand{\Eq}{{\mathscr E}}
\newcommand{\id}{\operatorname{id}}
\newcommand{\Hom}{\operatorname{Hom}}
\newcommand{\Ga}{{\mathbb G}_a}
\newcommand{\Gm}{{\mathbb G}_m}
\newcommand{\Aff}{\operatorname{Aff}}
\newcommand{\End}{\operatorname{End}}
\newcommand{\GL}{\operatorname{GL}}
\newcommand{\SL}{\operatorname{SL}}
\newcommand{\PSL}{\operatorname{PSL}}

\newcommand{\im}{\operatorname{im}}

\newcommand{\rank}{\operatorname{rank}}
\newcommand{\C}{{\mathbb C}}
\newcommand{\Mat}{\operatorname{Mat}}

\newcommand{\tens}{\operatorname{tensor}}
\newcommand{\category}{{\mathcal C}}
\newcommand{\N}{{\mathbb N}}

\newcommand{\F}{{\mathbb F}}
\newtheorem{theorem}{Theorem}[section]
\newtheorem{lemma}[theorem]{Lemma}
\newtheorem{proposition}[theorem]{Proposition}
\newtheorem{corollary}[theorem]{Corollary}
\newtheorem{problem}[theorem]{Problem}
\newcommand{\Z}{{\mathbb Z}}
\theoremstyle{definition}
\newtheorem{definition}[theorem]{Definition}
\newtheorem{example}[theorem]{Example}
\newtheorem{remark}[theorem]{Remark}
\newtheorem{algorithm}[theorem]{Algorithm}

\title[The Graph Isomorphism Problem and Approximate Categories]{The Graph Isomorphism Problem\\ and Approximate Categories}
\author{Harm Derksen}
\thanks{The author is supported by the NSF, grant DMS 0901298}
\begin{document}
\maketitle

\begin{abstract}
It is unknown whether two graphs can be tested for isomorphism in polynomial time. A classical approach
to the Graph Isomorphism Problem is the $d$-dimensional Weisfeiler-Lehman algorithm.  The $d$-dimensional
WL-algorithm can distinguish many pairs of graphs, but the pairs of non-isomorphic graphs constructed by Cai, F\"urer and Immerman
 it cannot distinguish.
If $d$ is fixed, then the WL-algorithm runs in polynomial time.
We will formulate the Graph Isomorphism Problem as an Orbit Problem: Given a representation $V$ of
an algebraic group $G$ and two elements $v_1,v_2\in V$, decide whether $v_1$ and $v_2$
lie in the same $G$-orbit. Then we attack the Orbit Problem by constructing certain approximate categories
${\mathcal C}_d(V), d\in \N=\{1,2,3,\dots\}$ whose objects include the elements of $V$. We show that $v_1$ and $v_2$
are not in the same orbit by showing that they are not isomorphic in the category ${\mathcal C}_d(V)$ for some $d\in \N$.
For every $d$ this gives us an algorithm for isomorphism testing. We will show that the WL-algorithms
reduce to our algorithms, but that our algorithms cannot be reduced to the WL-algorithms. 
Unlike the Weisfeiler-Lehman algorithm, our algorithm can  distinguish the Cai-F\"urer-Immerman graphs in polynomial time.
\end{abstract}
\section{Introduction and Main Results}
\subsection{The Graph Isomorphism Problem}
Suppose that $\Gamma_1$ and $\Gamma_2$ are two graphs on $n$ vertices.
The {\em Graph Isomorphism Problem} asks whether they are isomorphic or not.
In Computational Complexity Theory, the Graph Isomorphism Problem plays an important role, because
it lies in the complexity class {\bf NP}, but it is not known whether it lies in {\bf P} or {\bf NP-complete}.
See \cite{KUJ} for more details. Based on Valiant's algebraic version of the {\bf P} versus {\bf NP} problem (\cite{V}), 
Mulmuley and Sohoni reformulated Valiant's {\bf P} versus {\bf NP} problem into a question
about orbits of algebraic groups in \cite{MS,MS2}. In this paper, we will study the Graph Isomorphism in terms of orbits of algebraic groups, but our approach is not closely related to the work of Mulmuley and Sohoni.
%This means that there is no known algorithm which can decide in polynomial
%time whether $\Gamma_1$ and $\Gamma_2$ are isomorphic. However, if
%an explicit bijection between the vertices of $\Gamma_1$ and the vertices of $\Gamma_2$ is given, it
%can be verified in polynomial time that the bijection induces an isomorphism of graphs.
%It is also not known how reduce other {\bf NP}-problems to the graph isomorphism problem in polynomial time.

For special families of graphs there are polynomial time algorithms for the graph isomorphism problem. Polynomial
time algorithms were found for trees (Edmonds' algorithm, see \cite[p.196]{BS}), planar graphs (\cite{HT,HW}) and  more generally for
graphs of bounded genus (\cite{FM,Miller}), for graphs with bounded degree (\cite{Luks}),
for graphs with bounded eigenvalue multiplicity (\cite{BGM}), and for graphs with bounded color class size (\cite{Luks2}).

A general approach to the Graph Isomorphism Problem was developed by Weisfeiler and Lehman in the 1960's.
The $d$-dimensional Weisfeiler-Lehman
algorithm ${\bf WL}_d$ systematically colors $e$-tuples of vertices ($e\leq d$) until a stable coloring is obtained (see~\cite{WL,Weis}). 
%The $1$-dimensional Weisfeiler-Lehman algorithm works as follows. First we color the vertices of the graph
%according to their degrees. Two vertices get the same color if and only if they have the same degree. In the
%next round, we look at the colors of the neighbors of each vertex. Two vertices get the same color
%in this round if and only if in the previous round they had the same color
%and they had the same number of neighbors of  each given color.
%The new coloring is a refinement of the previous round. We repeat the process until the coloring
%in two consecutive rounds is the same. The $d$-dimensional Weisfeiler-Lehman algorithm is
%similar, but one colors $d$-tuples of vertices instead of just vertices.
The $d$-dimensional WL-algorithm terminates with a proof that the two graphs are isomorphic, or it terminates with an inconclusive result.
If $d\geq n$, then the $d$-dimensional Weisfeiler-Lehman algorithm will distinguish all non-isomorphic graphs with $n$ vertices.
For fixed $d$, the Weisfeiler-Lehman algorithm runs in polynomial time. The higher dimensional Weisfeiler-Lehman
algorithm can distinguish graphs in many families of graphs.  However, Cai, F\"urer and Immerman
showed in \cite{CFI} that for every $d$, there exists a pair of non-isomorphic graphs with degree $3$ and $O(d)$ vertices which cannot be distinguished by the
$d$-dimensional Weisfeiler-Lehman algorithm. The  set of Weisfeiler-Lehman algorithms ${\bf WL}=\{{\bf WL}_d\}_{d\in \N}$ is an example of what we will
call a {\em family of GI-algorithms}:
\begin{definition}
A family of GI-algorithms is a collection of algorithms ${\bf A}=\{{\bf A}_d\}_{d\in \N}$ such that
\begin{enumerate}
\item The input of ${\bf A}_d$ consists of two graphs with the same number of vertices. The value of the output is either ``{\tt nonisomorphic}'' or ``{\tt inconclusive}''.
If the output is ``{\tt nonisomorphic}'' then the graphs are not isomorphic and we say that ${\bf A}_d$ distinguishes the two graphs.
\item If the graphs are not isomorphic, then ${\bf A}_d$ distinguishes them some $d$.
\item For fixed $d$, ${\bf A}_d$ runs in polynomial time.
\end{enumerate}
\end{definition}
Besides the Weisfeiler-Lehman algorithm, there are other families of polynomial time algorithms for the graph isomorphism problem. In order to compare
various algorithms, we make the following definition (see also~\cite[\S6]{EKP}):
\begin{definition}
For two families of GI-algorithms ${\bf A}=\{{\bf A}_d\}_{d\in \N}$ and ${\bf B}=\{{\bf B}_d\}_{d\in\N}$ we 
say that ${\bf A}$ is reducible to ${\bf B}$
if there exists a function $f$ from the $\N\to\N$  such that for every $d$ and every pair of graphs which ${\bf A}_d$ distinguishes,
the graphs can be distinguished by ${\bf B}_{f(d)}$. We say that ${\bf A}$ and ${\bf B}$ are equivalent if ${\bf A}$ is reducible to ${\bf B}$ and ${\bf B}$ is reducible to ${\bf A}$.
\end{definition}

The Weisfeiler-Lehman algorithm is combinatorial in nature. There are also more algebraic approaches
to the graph isomorphism problem.
The $2$-dimensional
Weisfeiler-Lehman algorithm can be formulated in terms of {\em cellular algebras} (see~\cite{Weis}).\footnote{
These cellular algebras should not be confused with a different, seemingly unrelated
notion of cellular algebras introduced in \cite{GL}.} These algebras were introduced by Weisfeiler and Lehman,
and independently by D.~Higman under the name {\em coherent algebras} (see~\cite{Higman}).
%\begin{definition}
%A cellular algebra is a subalgebra ${\mathcal A}$ of the set $\Mat_{n,n}(\C)$ of $n\times n$ matrices such that
%\begin{enumerate}
%\item ${\mathcal A}$ contains the identity matrix and the matrix with all $1$'s;
%\item ${\mathcal A}$ is closed under taking the conjugates, if $B\in {\mathcal A}$, then $B=B^\star=\overline{B}^t\in {\mathcal A}$;
%\item ${\mathcal A}$ is closed under the Hadamard product: if $B=(b_{i,j})_{i,j=1}^n, C=(c_{i,j})_{i,j=1}^n\in {\mathcal A}$,
%then $B\circ C=(b_{i,j}c_{i,j})_{i,j=1}^n\in {\mathcal A}$.
%\end{enumerate}
%\end{definition}
%For every subset $S\subseteq \Mat_{n,n}(\C)$ there is a smallest cellular algebra containing $S$. We call this
%the closure of $S$.
%A cellular algebra is called {\em Schurian} if it is the centralizer algebra of a subgroup
%of the group of permutation matrices.
% To every graph with $n$ vertices, we can associate a cellular algebra by taking
%the closure of the set consisting of the adjacency matrix of the graph. To test whether two
%graphs are isomorphic we can test if the corresponding cellular algebras are isomorphic.
%This procedure is equivalent to the 2-dimensional Weisfeiler-Lehman algorithm.
%It is possible that two non-isomorphic graphs have isomorphic cellular algebras.
In \cite{EKP},
Evdokimov, Karpinski and Ponomarenko introduced the $d$-closure of
a cellular algebra. One may view the $d$-closed cellular algebras as 
higher-dimensional analogs of the cellular algebras.
The algorithm based on this $d$-closure will be denote by ${\bf CA}_d$. 
In \cite{EKP} it was shown that the algorithm ${\bf CA}_d$ distinguishes any two graphs which can be distinguished by ${\bf WL}_d$.
In \cite[Theorem 1.4]{EP}  it was shown that ${\bf WL}_{3d}$ can distinguish any two graphs which can be distinguished by ${\bf CA}_d$.
So the approach with cellular algebras ${\bf CA}=\{{\bf CA}_d\}_{d\in \N}$ is equivalent to the Weisfeiler-Lehman algorithm ${\bf WL}=\{{\bf WL}_d\}_{d\in \N}$.

In this paper we will define a family of GI-algorithms ${\bf AC}_d=\{{\bf AC}_d\}_{d\geq 0}$ using approximate categories.
We will show that ${\bf WL}$ reduces to ${\bf AC}$ and that ${\bf AC}$ does {\em not} reduce to  ${\bf WL}$.

\subsection{Finite variable logic}  Pairs of non-isomorphic graphs that can be distinguished with the Weisfeiler-Lehman algorithm can be characterized in terms of finite variable logic.

A graph is a pair $\Gamma=\langle X,R\rangle $ where $X$ is a finite set, and $R\subseteq X\times X$ is a symmetric relation.
We assume that there are no loops. We will write $xRy$ if $(x,y)\in R$.
%The edge set $E$ is defined by
%$$E=\{\{x,y\}\mid x,y\in X,\ (x,y)\in R\}.$$ 
We also consider graphs with colored vertices. A graph with $m$ colors is
a tuple $\langle X,R,X_1,X_2,\dots,X_m\rangle$ where $\langle X,R\rangle$ is a graph, and $X$ is the disjoint union
of $X_1,\dots,X_m$. Two colored graphs $\Gamma=\langle X,R,X_1,\dots,X_m\rangle $ and $\Gamma'=\langle X',R',X_1',\dots,X_m'\rangle$
are isomorphic if there is a bijection $\phi:X\to X'$ such that $x\in X_i\Leftrightarrow \phi(x)\in X_i'$
for all $x\in X$ and all $i$, and $xRy\Leftrightarrow \phi(x) R'\phi(y)$ for all $x,y\in X$.

We will view a graph with $m$ colors as a structure with 1 binary relation and $m$ unitary relations.
To such a structure one can associate a first order language $\LL$. If $\varphi$ is a closed formula
in $\LL$, then we will write $\Gamma\models \varphi$ if the formula $\varphi$ is true for $\Gamma$.

Let $\LL_d$ be the $d$-variable
first order language. Formulas in $\LL_d$ involve at most $d$ variables, although
variables may be re-used. For example
$$
\varphi(x_1,x_2)=\exists x_3\,(\exists x_2\, (x_1Rx_2 \wedge x_2Rx_3)\wedge x_3Rx_2)
$$
is a formula in $\LL_3$ which expresses that there exists a path of length $3$ from $x_1$ to $x_2$.
Note that in this formula, we re-use the variable $x_2$.

A more expressive language is $\CC_d$, the $d$-variable
first order language with counting. In this language we allow quantifiers such as $\exists_d$.
A formula
$\exists_d x\, \varphi(x)$
is true if there are exactly $d$ elements $x\in X$ for which $\varphi(x)$ is true.
%We define an equivalence relation $\sim_d$ on $m$-colored graphs  by 
%$\Gamma\sim_d \Gamma'$ if and only if
%for  every closed formula $\varphi\in \CC_d$, 
%$$
%\Gamma\models \varphi\Leftrightarrow \Gamma'\models \varphi
%$$
\begin{definition}
We say that the language $\CC_d$ distinguishes the colored graphs $\Gamma$ and $\Gamma'$ if there
exists a closed formula $\varphi$ in $\CC_d$ such that $\Gamma\models \varphi$
and $\Gamma'\models\neg\varphi$.
\end{definition}
\begin{theorem}[\S5 of \cite{CFI}]\label{theo:CFI}
The language $\CC_d$ distinguishes the colored graphs $\Gamma$ and $\Gamma'$ if and only
if the $(d-1)$-dimensional Weisfeiler-Lehman algorithm distinguishes the two graphs.
\end{theorem}

\subsection{Orbit problems}\label{sec:orbit_problems}
%In this paper, we will not only consider the graph isomorphism problem, but other isomorphism problems as well.
Fix a field $k$, and let $\overline{k}$ be its algebraic closure. 
Suppose that $G$ is an algebraic group defined over $k$, and $V$ is a representation of $G$ (over $k$).
Let $G(\overline{k})$ be the set of $\overline{k}$-rational points of $G$.

\medskip
\noindent
{\bf Orbit Problem:} Given $v_1,v_2\in V$, determine whether $v_1$ and $v_2$ lie in the same $G(\overline{k})$-orbit.

\medskip
\noindent
Many isomorphism problems can be translated to orbit problems. The graph isomorphism problem is one example of this. 
Let $G=\Sigma_n$ be the group of $n\times n$ permutation
matrices, and $V=\Mat_{n,n}(k)$ be the set of $n\times n$ matrices. Then $V$ is a representation of $\Sigma_n$
where $\Sigma_n$ acts by conjugation. To a graph $\Gamma$  with vertex set $\{1,2,\dots,n\}$ we can associate its adjacency matrix $A_{\Gamma}$ defined by
$$
(A_{\Gamma})_{i,j}=\left\{
\begin{array}{ll}
1 & \mbox{if there is an edge between $i$ and $j$;}\\
0 & \mbox{otherwise.}
\end{array}\right.
$$
The following lemma is obvious:
\begin{lemma}\label{lem:graph_orbit}
Two graphs $\Gamma_1$ and $\Gamma_2$ are isomorphic
if and only if their adjacency matrices $A_{\Gamma_1}$ and $A_{\Gamma_2}$ lie in the same $\Sigma_n$-orbit.
\end{lemma}
By replacing $V$ by a slightly different representation of $\Sigma_n$, one can generalize Lemma~\ref{lem:graph_orbit}
to colored graphs, and even to finite structures.
In Section~\ref{sec:module_iso} we will also translate the module isomorphism problem to an orbit problem.

\subsection{Main results}
In this paper we attack the orbit problem as follows. 
Suppose that $V$ is a representation of $G$. Let $\Aff(V)$ denote the set of
affine subspaces of $V$. So $\Aff(V)$ contains the empty set and all subsets of $V$
of the form $v+Z$, where $v\in V$ and $Z\subseteq V$ is a subspace. We may view
$V$ as a subset of $\Aff(V)$ by identifying $v\in V$ with the subset $\{v\}\subseteq V$.
The group $G$ acts on $\Aff(V)$. We also may view $\Aff(V)$ as a subset of $\Aff(V\otimes_k\overline{k})$
by identifying $v+Z\in \Aff(V)$ with $v\otimes 1+Z\otimes \overline{k}\in \Aff(V\otimes_k\overline{k})$
for every $v\in V$ and every subspace $Z\subseteq V$.
For every $d$
we construct a category $\category_d(V)$. 
For $X_1,X_2\in \Aff(V)$, we will write $X_1\cong_d X_2$ if $X_1$ and $X_2$ are isomorphic 
in ${\mathcal C}_d(V)$.
The categories ${\mathcal C}_d(V)$ have the following properties:
 with the following properties:
\begin{enumerate}
\item The set of objects of $\category_d(V)$ is $\Aff(V)$. In particular, elements of $V$ are objects in $\category_d(V)$.

\item $\category_d(V)$ is a $k$-category, i.e., for every two objects $X_1,X_2$ the set $\Hom_d(X_1,X_2)$ is a $k$-vector space,
and if $X_3$ is another object, then the composition map 
$$\Hom_d(X_1,X_2)\times \Hom_d(X_2,X_3)\to \Hom_d(X_1,X_3)$$ 
is bilinear.
\item For $X_1,X_2\in \Aff(V)$  we have $X_1\cong_{d+1}X_2\Rightarrow X_1\cong_d X_2$.
\item Two affine subspaces $X_1,X_2\in \Aff(V)$ lie in the same $G(\overline{k})$-orbit if and only if
$X_1\cong_d X_2$ for all $d$.
\end{enumerate}

An {\em equivariant} $f:V\to V'$ is a polynomial map between two representations which is $G$-equivariant.
An equivariant $f:V\to V'$ for which $V'$ is an irreducible representation is called a {\it covariant}.
If $k$ is algebraically closed, then $f$ being equivariant means that $f(g\cdot v)=g\cdot f(v)$
for all $v\in V$ and all $g\in G$. In the case where $k$ is not algebraically closed, equivariance is defined in Definition~\ref{def2.10}. We say that an equivariant $f:V\to V'$ distinguishes
two elements $v_1,v_2\in V$ if either $f(v_1)=0$ and $f(v_2)\neq 0$, or, $f(v_1)\neq 0$ and $f(v_2)=0$.
It is well known that if $v_1,v_2\in V$ are not in the same $G(\overline{k})$-orbit, then they can be distinguished by some equivariant.

For a representation $V$, $\Eq(V)$ denotes the class
of all equivariants $f:V\to V'$, where $V'$ is any representation.
For every positive integer $d$, we will define in Section~\ref{sec:constructible_equivariants} a subset $\Eq_d(V)\subseteq \Eq(V)$.
Elements of $\Eq_d(V)$ are called {\em $d$-constructible equivariants}.
We have
$$
\Eq_1(V)\subseteq \Eq_2(V)\subseteq \Eq_3(V)\subseteq \cdots
$$
and $\bigcup_{d=1}^\infty \Eq_d(V)=\Eq(V)$.
%\begin{definition}
%We will say that $\Eq_d(V)$ distinguishes $v,w\in V$ if there exists
%a $d$-constructible equivariant $f:V\to V'$ such that either $f(v)=0$ and $f(w)\neq 0$, or,
%$f(v)\neq 0$ and $f(w)=0$.
%\end{definition}

For a representation $V$ and a positive integer $d$ we will define in Section~\ref{sec:constructible} a class ${\mathcal F}_d(V)$ of 
functors ${\mathcal F}:{\mathcal C}_d(V)\to {\mathcal C}_d(V')$ where $V'$ is some representation.
Elements of ${\mathcal F}_d(V)$ are called {\em $d$-constructible functors}. We have
$$
\FF_1(V)\subseteq \FF_2(V)\subseteq \FF_3(V)\subseteq \cdots.
$$
The constructible functors are more general than the constructible equivariants in the following sense:
If $f:V\to V'$ is a $d$-constructible equivariant, then there exists a constructible functor ${\mathcal F}:{\mathcal C}_d(V)\to {\mathcal C}_d(V')$ such that ${\mathcal F}(\{v\})=\{f(v)\}$ for all $v\in V$ (see Lemma~\ref{lem:make_functor}).
%We will also define a class of functors ${\mathcal F}_d(V)$ of $d$-constructible functors
%${\mathcal F}:{\mathcal C}_d(V)\to {\mathcal C}_d(V')$. 
%\begin{definition}
%For $X_1,X_2\in \Aff(V)$, we say that $\FF_d(V)$ distinguishes $X_1,X_2$ if
%there exists a functor ${\mathcal F}:{\mathcal C}_d(V)\to {\mathcal C}_d(V')$
%such that $\dim {\mathcal F}(X_1)\neq \dim {\mathcal F}(X_2)$.
%\end{definition}
We will say that a functor ${\mathcal F}:{\mathcal C}_d(V)\to {\mathcal C}_d(V')$ distinguishes
$X_1,X_2\in \Aff(V)$ if $\dim {\mathcal F}(X_1)\neq \dim {\mathcal F}(X_2)$. Here, we use the
convention $\dim( \emptyset)=-\infty$.
\begin{theorem}\label{theo:general}
Suppose that $X_1,X_2\in \Aff(V)$.
We have the following implications:\\[10pt]
$\displaystyle
\begin{array}{lc}
\mbox{(i)} & \mbox{the $d$-constructible functors distinguish $X_1$ and $X_2$}\\
& \Downarrow\\
\mbox{(ii)} & \mbox{$X_1\not\cong_d X_2$}\\
& \Downarrow\\
\mbox{(iii)} & \mbox{$X_1$ and $X_2$ lie in distinct $G(\overline{k})$-orbits}
\end{array}
$
\end{theorem}
\begin{theorem}\label{theo:equis}
Suppose that $v_1,v_2\in V$. We have the following implications:\\[10pt]
$\displaystyle
\begin{array}{lc}
\mbox{(i)} & \mbox{the $d$-constructible equivariants distinguish $v_1$ and $v_2$}\\
& \Downarrow\\
\mbox{(ii)} & \mbox{the $d$-constructible functors distinguish $v_1$ and $v_2$}\\
& \Downarrow\\
\mbox{(iii)} & \mbox{$v_1\not\cong_dv_2$}\\
& \Downarrow\\
\mbox{(iv)}& \mbox{$v_1$ and $v_2$ lie in distinct $G(\overline{k})$-orbits}
\end{array}
$
\end{theorem}
The following proposition follows from Proposition~\ref{prop:counterexample} and 
shows that the implication (i)$\Rightarrow$(ii) in Theorem~\ref{theo:equis} cannot be reversed.

\begin{proposition}\label{prop:not_converse}
For every $d$ there exist examples, where $v_1$
and $v_2$ can be distinguished by $3$-constructible functors, but not by $d$-constructible equivariants.
\end{proposition}

If we reformulate the Graph Isomorphism Problem as an Orbit Problem as in Section~\ref{sec:orbit_problems} we get the following results.
\begin{theorem}\label{theo:graph_implications}
Suppose that $k$ is a field of characteristic 0 or characteristic $p$ with $p>n$.
Suppose that $\Gamma_1,\Gamma_2$ are graphs on $n$-vertices, or more generally,
structures on sets with $n$ elements. We have the following implications:\\[10pt]
$\displaystyle
\begin{array}{lc}
\mbox{(i)} & \mbox{the language $\LL_d$ distinguishes $\Gamma_1$ and $\Gamma_2$}\\
& \Downarrow\\
\mbox{(ii)} &\mbox{the language $\CC_d$ distinguishes $\Gamma_1$ and $\Gamma_2$}\\
& \Updownarrow\\
\mbox{(iii)} &\mbox{the $(d-1)$-dim. Weisfeiler-Lehman algorithm distinguishes $\Gamma_1$ and $\Gamma_2$}\\
& \Downarrow\\
\mbox{(iv)} &\mbox{the $(2d)$-constructible equivariants distinguish $A_{\Gamma_1}$ and $A_{\Gamma_2}$}\\
& \Downarrow\\
\mbox{(v)} &\mbox{the $(2d)$-constructible functors distinguish $A_{\Gamma_1}$ and $A_{\Gamma_2}$}\\
& \Downarrow\\
\mbox{(vi)} &\mbox{$A_{\Gamma_1}\not\cong_{2d}A_{\Gamma_2}$}\\
& \Downarrow\\
\mbox{(vii)} &\mbox{$A_{\Gamma_1}$ and $A_{\Gamma_2}$ do not lie in the same $\Sigma_n$-orbit}\\
& \Updownarrow\\
\mbox{(viii)} &\mbox{$\Gamma_1\not\cong\Gamma_2$}
\end{array}
$
\end{theorem}
%\begin{theorem}
%Suppose that $k=\F_p$, where $p$ is a prime $>n$, and $A_\Gamma$ and $A_{\Gamma'}$ are isomorphic in $\category_d$.
%Then $\Gamma$ and $\Gamma'$ cannot be distinguished with the $d$-dimensional Weisfeiler-Lehman algorithm.
%\end{theorem}
\begin{theorem}\label{theo:polytime}
Suppose that $V=\Mat_{n,n}(\F_p)$ is a representation of $\Sigma_n$ over $\F_p$ as in Section \ref{sec:orbit_problems}, where $p=p(n)$ is a prime for all $n$  and $\log p(n)$ grows polynomially.
Then one can check whether two objects in ${\mathcal C}_d(V)$ are isomorphic in polynomial time.
\end{theorem}
The proofs of Theorem~\ref{theo:general}, \ref{theo:equis}, \ref{theo:graph_implications}  and \ref{theo:polytime} are in Section~\ref{sec:cat_iso}.

Because of Proposition~\ref{prop:not_converse} and Theorem~\ref{theo:graph_implications} one might believe that constructible functors are more powerful
in distinguishing graphs than the Weisfeiler-Lehman algorithm. The following theorem shows
that this is the case when working over the field $\F_2$:
\begin{theorem}\label{theo:CFInoniso}
If $k=\F_2$ and $\Gamma_1,\Gamma_2$ are nonisomorphic (colored) graphs constructed following the Cai-F\"urer-Immerman method,
then there exists a $3$-constructible functor which distinguishes $A_{\Gamma_1}$ and $A_{\Gamma_2}$. 
In particular $A_{\Gamma_1}$ and $A_{\Gamma_2}$ are not isomorphic in ${\mathcal C}_3(V)$.
\end{theorem}
The proof of Theorem~\ref{theo:CFInoniso} is in Section~\ref{sec:CFI}.

\begin{algorithm}[The algorithm ${\bf AC}_d$]\label{alg}
To check whether two graphs $\Gamma_1$ and $\Gamma_2$ on $n$ vertices are isomorphic, we can determine
whether $A_{\Gamma_1}$ and $A_{\Gamma_2}$ are isomorphic in ${\mathcal C}_{d}(V)$,
where we work over the field $\F_p$ and $p$ runs over all primes $\leq 2n$. 
\end{algorithm}
Note
that there exists a prime between $n$ and $2n$ by Bertrand's postulate (see~\cite{Rama}).
By Theorem~\ref{theo:polytime}, ${\bf WL}$ reduces to ${\bf AC}$, because ${\bf AC}_{2d+2}$ distinguishes any two graphs
that are distinguished by ${\bf WL}_d$. On the other hand, for every $d$,  Cai-F\"urer and Immerman constructed  pairs of graphs that cannot be
distinguished by ${\bf WL}_d$. By Theorem~\ref{theo:CFInoniso}, these graphs are distinguished by ${\bf AC}_3$. This shows that
${\bf AC}$ cannot be reduced to ${\bf WL}$.

\section{The construction of the approximate categories}
\subsection{Truncated ideals}
To define the approximate categories, we will need the notion of a truncated ideal. 
Suppose that $k$ is a field, and $R$ is a finitely generated $k$-algebra
with a filtration
$$
R_0\subseteq R_1\subseteq R_2\subseteq \cdots
$$
such that $R_0=k$ and $R_i$ is a finite-dimensional vector space for all $i$.
\begin{definition}
If $S\subseteq R_d$ then we define 
$$(S)_d=\sum_{e=0}^d(S\cap R_e)R_{d-e}.
$$
\end{definition}
\begin{definition}
A subset $S\subseteq R_d$ is called a {\em $d$-truncated ideal} if $(S)_d=S$.
\end{definition}
We have a chain
$$
(S)_d\subseteq ((S)_d)_d\subseteq (((S)_d)_d)_d\subseteq \cdots.
$$
Since $R_d$ is finite dimensional, this chain stabilizes to a
subspace of $R_d$ which we will denote by $((S))_d$.
It is clear that $((S))_d$ is the smallest $d$-truncated ideal
containing $S$. We will call it the {\em $d$-truncated ideal generated by $S$}.
\begin{example}
 Consider the polynomial ring $k[x,y]$ in two variables with the usual grading.
We have 
$$y-x^2=-x(x-y^2)-y(xy-1)
\in (x-y^2,xy-1),$$
but $y-x^2\not\in ((x-y^2,xy-1))_2$.
\end{example}
\begin{remark}
Much of the theory of Gr\"obner basis generalizes to truncated ideals. 
Suppose that $R=k[x_1,\dots,x_n]$ is the polynomial ring.
In a polynomial ring we can 
choose a monomial ordering which is compatible with the grading: if one monomial has higher degree than another
monomial, then it is larger in the monomial ordering. A subset ${\mathscr G}$ of a $d$-truncated ideal $J$
is a truncated Gr\"obner bases if the ideal generated by the leading monomials of elements of $J$ is
the same as the ideal generated by the leading monomials of elements of ${\mathscr G}$. There is
also an analog of Buchberger's algorithm. Starting with a set of generators of $J$, one obtains
a truncated Gr\"obner bases by reducing S-polynomials whose total degree is $\leq d$.
Since $R_d$ is a finite dimensional vector space, computations with truncated ideals can be done
by just using linear algebra. However, using truncated Gr\"obner bases exploits the ring structure
and may speed up the computations. For complexity bounds, the linear algebra approach will
be good enough, so we will not explore the truncated Gr\"obner bases in detail here.
\end{remark}

\begin{proposition}\label{prop:3}
Suppose that $R_e=R_1^e$ for all $e\geq 1$, i.e., $R_e$ is spanned by all products
$f_1f_2\cdots f_e$ with $f_1,\dots,f_e\in R_1$.
There exists a constant $C(d)$ (depending on $d$, $R$ and the filtration)
such that $((S))_e=(S)\cap R_e$ for all $e\geq C(d)$ and all $S\subseteq R_d$.
\end{proposition}
\begin{proof}
Define a ring homomorphism
$$
\gamma:k[x_1,\dots,x_n]\to R.
$$
such that $\gamma(x_1),\dots,\gamma(x_n)$ span $R_1$.
Suppose that $h_1,\dots,h_r\in k[x_1,\dots,x_n]$ generate the kernel of $\gamma$,
and let $l$ be the maximum of the degrees of $h_1,\dots,h_r$.
Assume that $S\subseteq R_d$ is a subspace spanned by $f_1,\dots,f_s$.
For all $i$, choose $\widetilde{f}_i\in k[x_1,\dots,x_n]$ of degree $\leq d$
with $\gamma(\widetilde{f}_i)=f_i$.
Let $J$ be the ideal generated by the set ${\mathcal G}=\{\widetilde{f}_1,\dots,\widetilde{f}_s,h_1,\dots,h_r\}$.
Then we have $\gamma(J)=(S)$.
We can apply Buchberger's algorithm to the generator set ${\mathcal G}$ to obtain
a Gr\"obner basis for the ideal $J$.  It was shown in \cite{W} that there exists a universal bound $C(d)$,
(depending only on $d$, $n$, and $l$) such that all polynomials in the reduced Gr\"obner basis,
and all polynomials appearing in intermediate steps of  Buchberger's algorithm
have degree $\leq C(d)$.  Following the Buchberger algorithm, it is easy to see
that $\gamma(h)\in ((S))_{C(d)}$ for all elements $h$ in the Gr\"obner basis of $J$.
Suppose that $f\in (S)\cap R_e$ and $e\geq C(d)$. We can lift $f$ to an element $\widetilde{f}$
such that $\deg(\widetilde{f})\leq e$ and $\gamma(\widetilde{f})=f$.
We can write $\widetilde{f}=\sum_{i=1}^t a_iu_i$ where $u_1,\dots,u_t$ are elements
of the Gr\"obner basis, and $\deg(a_iu_i)\leq e$ for all $i$.
From this follows that
$$
f=\gamma(\widetilde{f})=\sum_{i=1}^t\gamma(a_i)\gamma(u_i)\in \sum R_{e-j}\big( ((S))_e\cap R_{j}\big)=(((S))_e)_e=((S))_e.
$$
\end{proof}

\subsection{Coalgebras associated to algebraic groups}
Suppose that $G$ is a linear algebraic group over $k$.
Let $R:=k[G]$ be the coordinate ring of $G$. The identity element $e\in G$
is defined over the field $k$.
The multiplication $G\times G\to G$ corresponds to a homomorphism
of $k$-algebras
$$
\Delta:R\to R\otimes R,
$$
where $\otimes$ denotes the tensor product as $k$-vector spaces.
The ring $R$ is a Hopf algebra with co-multiplication $\Delta$,
and counit $\sigma_e:R\to k$, where $\sigma_e$ is evaluation at $e\in G$.
The inverse function $G\to G$ defined by $g\mapsto g^{-1}$ defines a antipode
map $\iota:R\to R$. 

Suppose for the moment that $k$ is algebraically closed.
The group $G$ acts on itself by left multiplication
and it acts on the right by right multiplication.
These actions correspond to a left and right action
of $G$ on $R$. If $g\in G$, then $g$ acts on $R$ on the right as
the automorphism
$$
(\sigma_g\otimes \id)\circ \Delta:R\to R
$$
where $\sigma_g:R\to k$ is evaluation at $g$.
The element $g$ acts on the left by the automorphism
$$
(\id\otimes \sigma_g)\circ \Delta:R\to R.
$$
A subspace $W\subseteq R$ is stable under the right action if
$$
\Delta(W)\subseteq R\otimes W
$$
and stable under the left action if
$$
\Delta(W)\subseteq W\otimes R.
$$
It is stable under both actions if
$$
\Delta(W)\subseteq W\otimes W.
$$

Let $k$ again be an arbitrary field, and
suppose that $W\subseteq R$ is a subspace such that
\begin{enumerate}
\item $W$ contains $k$;
\item $W$ generates $R$;
\item $\Delta(W)\subseteq W\otimes W$;
\item $\iota(W)\subseteq W$.
\end{enumerate}
We will call such a subspace $W$ a {\em stable generating subspace}.
We define a filtration on $R$ by $R_0:=k$ and $R_d:=W^d$ for $d>0$. 
We have $\Delta(R_d)\subseteq R_d\otimes R_d$ and $\iota(R_d)\subseteq R_d$, so $R_d$ is a co-associative co-algebra.
The dual $R_d^\star$ of $R_d$ is an associative algebra.
\begin{example}\label{ex:Ga}
 Suppose that $G=\Ga=(k,+)$ is the additive group. Then $k[\Ga]$ is isomorphic
to $k[t]$, the polynomial ring in one variable. The identity element is $e=0\in k$.
So the co-unit is $\sigma_0$, which is defined by:
$$
\sigma_0(f(t))=f(0).
$$
The co-multiplication 
$$
\Delta:k[t]\to k[t]\otimes k[t]
$$
is defined by
$$
\Delta(f(t))=f(t\otimes 1+1\otimes t).
$$
and the $\iota:k[t]\to k[t]$ is defined by
$$
\iota(f(t))=f(-t).
$$
We can take $W=k\oplus k\cdot t\subseteq k[t]$.
Then we have $k\subseteq W$, $W$ generates $k[t]$, $\Delta(W)\subseteq W\otimes W$ and $\iota(W)\subseteq W$.
Now $W^d\subseteq k[t]$ consists of all polynomials of degree $\leq d$. This is a natural
filtration on the ring $k[t]$.
\end{example}
\begin{example}\label{ex:Gm}
 Suppose that $G=\Gm=(k^\star,\cdot)$ is the multiplicative group.
Then $k[\Gm]$ is isomorphic to the ring $k[t,t^{-1}]$ of Laurent polynomials.
The identity element is $e=1\in \Gm$. So the co-unit is
$$
\sigma_1:k[t,t^{-1}]\to k
$$
defined by
$$
\sigma_1(f(t))=f(1).
$$
The co-multiplication $\Delta:k[t,t^{-1}]\to k[t,t^{-1}]\otimes k[t,t^{-1}]$ is defined by
$$
\Delta(f(t))=f(t\otimes t)
$$
and $\iota:k[t,t^{-1}]\to k[t,t^{-1}]$ is defined by
$$
\iota(f(t))=f(t^{-1}).
$$
Define $W\subseteq k[t,t^{-1}]$ by $W=kt^{-1}\oplus k\oplus kt$.
Then we have $k\subseteq W$, $W$ generates $k[t,t^{-1}]$, $\Delta(W)\subseteq W\otimes W$ and $\iota(W)\subseteq W$.
The space $W^d$ is the space of all Laurent polynomials of the form
$$
\sum_{i=-d}^d a_it^i
$$
where $a_{-d},a_{1-d},\dots,a_d\in k$.
\end{example}

\begin{lemma}\label{lem:ABC}
Suppose that $A,B,C$ are subspaces of $R_d$
with 
$$\Delta(A)\subseteq B\otimes R_d+R_d\otimes C.$$
Then we have
\begin{equation}\label{eq:ABC}
\Delta((A)_d)\subseteq (B)_d\otimes R_d+R_d\otimes (C)_d
\end{equation}
and
\begin{equation}\label{eq:ABC2}
\Delta(((A))_d)\subseteq ((B))_d\otimes R_d+R_d\otimes ((C))_d.
\end{equation}
\end{lemma}
\begin{proof}
The space
$$
R_e/(B\cap R_e)\otimes R_e/(C\cap R_e)=(R_e\otimes R_e)/((B\cap R_e)\otimes R_e+R_e\otimes (C\cap R_e))$$
is a subspace of
$$
R_d/B\otimes R_d/C=(R_d\otimes R_d)/(B\otimes R_d+R_d\otimes C).
$$
It follows that
\begin{multline*}
\Delta(A\cap R_e)\subseteq \Delta(A)\cap \Delta(R_e)\subseteq
(B\otimes R_d+R_d\otimes C)\cap (R_e\otimes R_e)=\\=
(B\cap R_e)\otimes R_e+R_e\otimes (C\cap R_e).
\end{multline*}
Therefore we have
\begin{multline*}
\Delta((A\cap R_e)R_{d-e})\subseteq
\Delta(A\cap R_e)\Delta(R_{d-e})=\\=
((B\cap R_e)\otimes R_e+R_e\otimes (C\cap R_e))(R_{d-e}\otimes R_{d-e})\subseteq\\
\subseteq
((B\cap R_e)R_{d-e})\otimes R_d+R_d\otimes ((C\cap R_e)R_{d-e}).
\end{multline*}
This shows (\ref{eq:ABC}). Now (\ref{eq:ABC2}) follows by iteration.
\end{proof}
%\begin{proof}
%From $\Delta(S)\subseteq S\otimes S$ and $\Delta(R_e)\subseteq R_e\otimes %R_e$ follows that
%$$
%%\Delta(S\cap R_e)\subseteq \Delta(S)\cap \Delta(R_e)=
%(S\otimes S)\cap (R_e\otimes R_e)=(S\cap R_e)\otimes (S\cap R_e).
%$$
%and therefore
%\begin{multline*}
%\Delta((S\cap R_e)R_{d-e})=
%\Delta(S\cap R_e)\Delta(R_{d-e})\subseteq\\
%\subseteq
%((S\cap R_e)\otimes (S\cap R_e))(R_{d-e}\otimes R_{d-e})
%=
%((S\cap R_e)R_{d-e})\otimes (S\cap R_e)R_{d-e}.
%\end{multline*}
%\end{proof}

\subsection{The complexity of a representation}
Let $G$ be a linear algebraic group over $k$ and fix a stable generating subspace $W$.
\begin{definition}
A rational representation of $G$
is a finite dimensional vector space $V$ with a $k$-linear map
$$
\mu: V\to V\otimes R
$$
such that the diagram
$$
\xymatrix{
V\ar[rr]^\mu\ar[d]_\mu && V\otimes R \ar[d]^{\mu\otimes \id}\\
V\otimes R\ar[rr]_{\id\otimes \Delta} & &V\otimes R\otimes R}
$$
commutes, and $(\id\otimes \sigma_e)\circ\mu=\id$.
\end{definition}
If $k$ is algebraically closed, then we define
$$
g\cdot w=(\id\otimes \sigma_g)\circ\mu(w)
$$
for all $g\in G$ and $w\in V$.
We have the following commutative diagram
$$
\xymatrix{
V\ar[rr]^\mu\ar[d]_\mu && V\otimes R\ar[rr]^{\id\otimes \sigma_h}\ar[d]_{\mu\otimes\id} && V\ar[d]^{\mu}\\
V\otimes R\ar[rr]^{\id\otimes \Delta}\ar[rrd]_{\id\otimes\sigma_{gh}} && V\otimes R\otimes R\ar[rr]^{\id\otimes\id\otimes \sigma_h}\ar[d]^{\id\otimes\sigma_g\otimes\sigma_h} & & V\otimes R\ar[lld]^{\id\otimes\sigma_g} \\
& &  V &&}.
$$
This shows that 
$$
(gh)\cdot v=(\id\otimes \sigma_{gh})\circ \mu(v)=(\id\otimes \sigma_g)\circ\mu\circ (\id\otimes\sigma_h)\circ \mu(v)=
g\cdot (h\cdot v).
$$
\begin{definition}\label{def2.10}
Suppose that $V$ and $V'$ are rational representations of $G$ given by $\mu:V\to V\otimes R$ and $\mu':V'\to V'\otimes R$.
A linear map $f:V\to V'$ is called $G$-equivariant if the following diagram commutes:
$$
\xymatrix{
V\ar[rr]^\mu\ar[d]_{f} & & V\otimes R\ar[d]^{f\otimes \id}\\
V'\ar[rr]_{\mu'} & & V'\otimes R}.
$$
\end{definition}
If $k$ is algebraically closed, $v\in V$ and $g\in G$ then we have
$$
f(g\cdot v)=f\circ (\id\otimes \sigma_g)\circ \mu(v)=(\id\otimes \sigma_g)\circ (f\otimes \id)\circ \mu=(\id\otimes \sigma_g)\circ \mu\circ f(v)=
g\cdot f(v).
$$

Assume that $\ell_W(V)$ is the smallest nonnegative integer such that 
$$\mu(V)\subseteq R_{\ell_W(V)}\otimes V.$$
The number $\ell_W(V)$ depends on the choice of $W$, but we will often
drop the subscript and just write $\ell(V)$ if $W$ is fixed.
We can think of $\ell(V)$ as a measure of the complexity of
the representation $V$.
\begin{lemma}\
\begin{enumerate}
\item $\ell(V\oplus V')=\max\{\ell(V),\ell(V')\}$;
\item $\ell(V\otimes V')\leq \ell(V)+\ell(V')$;
\item $\ell(V)=\ell(V^\star)$.
\end{enumerate}
\end{lemma}
\begin{proof}\ 

\noindent{\bf (1)} This is straightforward.

\noindent {\bf (2)} The representation $V$ and $V'$ are given by $\mu:V\to V\otimes R$ and $\mu':V'\to V'\otimes R$.
We have $\mu(V)\subseteq V\otimes R_{\ell(V)}$ and $\mu'(V')\subseteq V'\otimes R_{\ell(V')}$.
The representation $V\otimes V'$ is given by the composition $\mu''$ defined by:
$$
\xymatrix{
V\otimes V'\ar[rr]^{\mu\otimes \mu'}&&  V\otimes R\otimes V'\otimes R\ar[r]^{\cong} & V\otimes V'\otimes R\otimes R\ar[rr]^{\id\otimes\id\otimes m} && V\otimes V'\otimes R.
}
$$ 
where $m:R\otimes R\to R$ is the usual multiplication given by $\sum_{i}a_i\otimes b_i\mapsto \sum_i a_ib_i$.
We have 
$$\mu\otimes \mu'(V\otimes V')\subseteq V\otimes R_{\ell(V)}\otimes V'\otimes R_{\ell(V')},$$
so
$$
\mu''(V\otimes V')\subseteq (\id\otimes\id\otimes m)( V\otimes V'\otimes R_{\ell(V)}\otimes R_{\ell(V')})\subseteq
V\otimes V'\otimes R_{\ell(V)}R_{\ell(V')}\subseteq V\otimes V'\otimes R_{\ell(V)+\ell(V')}.
$$

\noindent 
{\bf (3)} Let $\mu^\star:V^\star \to V^\star\otimes R $ be the dual representation of $\mu:V\to V\otimes R$.
We have 
$$
(f\otimes \iota)\circ \mu(v)=(v\otimes \id)\circ \mu^\star(f)
$$
where $v\in V=V^{\star\star}$ and $f\in V^\star$.
If $\ell(V)=d$, then $\mu(v)\in V\otimes R_d$, and 
$$(v\otimes\id)\circ \mu^\star(f)=(f\otimes \iota)(\mu(v))\in( f\otimes \iota)(V\otimes R_d)\subseteq\iota(R_d)\subseteq R_d.
$$
It follows that $\mu^\star(f)\subseteq V^\star \otimes R_d$.
This shows that $\ell(V^\star)\leq \ell(V)\leq d$. Similarly, we have $\ell(V)=\ell(V^{\star\star})\leq \ell(V^{\star})$.
\end{proof}
\begin{example}\label{ex:Ga2}
Suppose we are in the context of Example~\ref{ex:Ga}.
Assume that $k$ is a field of characteristic $0$.
Let $V_d$ be the $(d+1)$-dimensional indecomposable representation of $\Ga$ ($d\geq 0$). We can choose a basis $x_0,x_1,\dots,x_d$
 of $V_d$ such that the action  $\mu_d:V_d\to V_d\otimes R$ is given by:
$$
\mu_d(x_i)=x_i\otimes 1+x_{i-1}\otimes t+x_{i-2}\otimes \frac{t^2}{2!}+\cdots+x_0\otimes \frac{t^i}{i!}.
$$
It follows that $\mu(V_d)\subseteq V_d\otimes R_d$ and $\mu_d(V_d)\not\subseteq V_d\otimes R_{d-1}$.
We conclude that $\ell(V_d)=d$.
If $V$ is any representation, then $V$ is of the form
$$
V=V_{d_1}\oplus \cdots \oplus V_{d_r}
$$
and 
$$
\ell(V)=\max\{d_1,\dots,d_r\}.
$$
\end{example}
\begin{example}\label{ex:Gm2}.
Suppose that we are in the setup of Example~\ref{ex:Gm}.
For $d\in \Z$, let $V_d\cong k$ be the irreducible $1$-dimensional representation of $\Gm$ defined by $\mu_d:V_d\to V_d\otimes R$,
where $\mu_d$ is given by
$$
\mu_d(1)=1\otimes t^d.
$$
Then we clearly have $\ell(V_d)=|d|$.
Since $V_d,d\in\Z$ are all irreducible representations, any representation $V$ can be written as
$$
V=V_{d_1}\oplus \cdots \oplus V_{d_r}.
$$
Then we have 
$$
\ell(V)=\max\{|d_1|,|d_2|,\dots,|d_r|\}.
$$
\end{example}

\subsection{Definition of the approximate categories}
For a subspace $Z\subseteq V$ we define $Z^\perp=\{f\in V^\star\mid \forall v\in V\,f(v)=0\}$.

Suppose that $X_1,X_2\in \Aff(V)$. If
$k$ is algebraically closed, then the equation
 $g\cdot X_1\subseteq X_2$ ($g\in G$) yields a
 system of polynomial equations in $k[G]$.
 If $X_1$ and $X_2$ are nonempty then we can write
 $X_1=v_1+Z_1$ and $X_2=v_2+Z_2$. Let $V^\star$ be the dual
 of $V$ and $Z_2^\perp$ be the space of all $f\in V^\star$ 
 which vanish on $Z_2$.
 For every function $f\in Z_2^\perp\subseteq V^\star$
 on $V$ vanishing on $Z_2$, and every $w\in X_1$
  we have the equation $f(g\cdot w)=f(v_2)$.
  In other words, 
  $$(f\otimes \id)\circ \mu(w)-f(v_2)\otimes 1=0.$$
  The latter equation makes sense, even if $k$ is not algebraically
  closed.
  \begin{definition}
  Let $S(X_1,X_2)$ be the span of all
  $$(f\otimes \id)\circ \mu(w)- f(v_2)\otimes 1=(f\otimes \id)(\mu(w)-f(v_2)\otimes 1)$$
  with $f\in Z_2^\perp$ and $w\in X_1$.
  We define $S(\emptyset,X)=\{0\}$ for $X\in \Aff(V)$
  and $S(X,\emptyset)=\{1\}$ if $X\in \Aff(V)\setminus\{\emptyset\}$.
  \end{definition}

\begin{lemma}\label{lem:Scoideal}
If $X_1,X_2,X_3\in \Aff(V)$, then we have
$$
\Delta(S(X_1,X_3))\subseteq S(X_2,X_3)\otimes R_{\ell(V)}+R_{\ell(V)}\otimes S(X_1,X_2).
$$
\end{lemma}
\begin{proof}
Suppose that $X_i=v_i+Z_i$ for $i=1,2,3$.
We have
$$
\mu(v_2)-v_3\otimes 1\in V\otimes S(X_2,X_3)+Z_3\otimes R_{\ell(V)},
$$
and
$$
\mu(Z_2)\subseteq V\otimes S(X_2,X_3)+Z_3\otimes R_{\ell(V)}.
$$
It follows that
\begin{multline*}
(\id\otimes \Delta)(\mu(v_1)-v_3\otimes 1)=
(\id\otimes \Delta)\circ \mu(v_1)-v_3\otimes 1\otimes 1=\\
=
(\mu\otimes \id)\circ\mu(v_1)-v_3\otimes 1\otimes 1=
(\mu\otimes \id)\circ (\mu(v_1)-v_2\otimes 1)+ (\mu(v_2)-v_3\otimes 1)\otimes 1\in\\
\in
(\mu\otimes \id)(V\otimes S(X_1,X_2)+Z_2\otimes R_{\ell(V)})+
V\otimes S(X_2,X_3)\otimes R_{\ell(V)}+Z_3\otimes R_{\ell(V)}\otimes R_{\ell(V)}\subseteq\\
\subseteq
V\otimes R_{\ell(V)}\otimes S(X_1,X_2)+V\otimes S(X_2,X_3)\otimes R_{\ell(V)}+Z_3\otimes R_{\ell(V)}\otimes R_{\ell(V)}.
\end{multline*}
If $f\in Z_3^\perp$, then we have
\begin{multline*}
\Delta\big(
(f\otimes \id)(\mu(v_1)-v_3\otimes 1)\big)=(f\otimes\id\otimes \id)\big(\id\otimes \Delta(\mu(v_1)-v_3\otimes 1)\big)\subseteq\\
\subseteq
(f\otimes\id\otimes \id)(V\otimes R_{\ell(V)}\otimes S(X_1,X_2)+V\otimes S(X_2,X_3)\otimes R_{\ell(V)}+Z_3\otimes R_{\ell(V)}\otimes R_{\ell(V)})\subseteq\\
\subseteq S(X_2,X_3)\otimes R_{\ell(V)}+R_{\ell(V)}\otimes S(X_1,X_2).
\end{multline*}
\end{proof}
For $X_1,X_2\in \Aff(V)$, define
$$
I_d(X_1,X_2)=((S(X_1,X_2)))_d
$$
for all $d\geq \ell(V)$. We also define
$$
I_{\infty}(X_1,X_2)=(S(X_1,X_2))\subseteq R.
$$
For $d\geq \ell(V)$ we have
$$
I_{d}(X_1,X_2)\subseteq I_{d+1}(X_1,X_2)\subseteq \cdots, 
$$
where $I_{\infty}(X_1,X_2)=\bigcup_{j\geq d} I_{j}(X_1,X_2)$.
For $d\geq \ell(V)$ we have a natural linear map
$\psi_d:R_d/I_d(X_1,X_2)\to R_{d+1}/I_{d+1}(X_1,X_2)$.
This gives us a  chain of linear maps
\begin{equation}\label{eq:directlimit}
R_d/I_d(X_1,X_2)\to R_{d+1}/I_{d+1}(X_2,X_3)\to R_{d+2}/I_{d+2}(X_1,X_2)\to\cdots
\end{equation}
There also is a natural linear map $\gamma_d:R_d/I_d(X_1,X_2)\to R/I_\infty(X_1,X_2)$ for all $d$. We have $\gamma_{d+1}\circ \psi_d=\gamma_d$
for all $d$.
By Proposition~\ref{prop:3}, there exists a constant $C=C(\ell(V))$ such that 
$\gamma_d$ is injective for large $d\geq C$.
This implies that $\psi_d$ is injective for large $d$.
This shows that $R/I_{\infty}(X_1,X_2)$ is the direct
limit the diagram (\ref{eq:directlimit}):

\begin{corollary}\label{cor:Icoideal}
For $X_1,X_2,X_3\in \Aff(V)$ we have
$$
\Delta(I_d(X_1,X_3))\subseteq I_d(X_2,X_3)\otimes R_d+R_d\otimes I_d(X_1,X_2).
$$
\end{corollary}
\begin{proof}
This follows from Lemma~\ref{lem:ABC} and Lemma~\ref{lem:Scoideal}.
\end{proof}
For $X_1,X_2\in \Aff(V)$ and $d\geq \ell(V)$, define
$$
\Hom_d(X_1,X_2)=(R_d/I_d(X_1,X_2))^\star.
$$
We also define
$$
\Hom_\infty(X_1,X_2)=(R/I(X_1,X_2))^\star
$$

It follows from the definitions and Hilbert's Nullstellensatz that
$$
\Hom_{\infty}(X_1,X_2)\neq 0\Leftrightarrow I_{\infty}(X_1,X_2)\neq R\Leftrightarrow\exists g\in G(\overline{k})\, g\cdot X_1\subseteq X_2.
$$

Now $\Hom_{\infty}(X_1,X_2)$ is the inverse limit
of the diagram
$$
\cdots \to \Hom_{d+2}(X_1,X_2)\to \Hom_{d+1}(X_1,X_2)\to \Hom_d(X_1,X_2).
$$
The map $\gamma_{d}^\star:\Hom_{\infty}(X_1,X_2)\to \Hom_d(X_1,X_2)$
is onto for $d\geq C$ large enough, where $C=C(\ell(V))$ is a constant depending on $\ell(V)$.
\begin{corollary}
There exists a constant $C$ such that for all $X_1,X_2\in \Aff(V)$ and $d\geq C$, we have
$$
\Hom_d(X_1,X_2)\neq 0\Leftrightarrow \exists g\in G(\overline{k})\,g\cdot X_1\subseteq X_2.
$$
and
$$
\begin{array}{c}
\mbox{$X_1,X_2$ isomorphic in ${\mathcal C}_d(V)$}\\
\Updownarrow\\
\mbox{$\Hom_d(X_1,X_2)\neq 0$ and $\Hom_d(X_2,X_1)\neq 0$}\\
\Updownarrow\\
\mbox{$X_1,X_2$ are in
the same $G(\overline{k})$-orbit}\end{array}.
$$

\end{corollary}

We can view $\Hom_d(X_1,X_2)$ as a subspace of $R_d^\star$.
From Corollary~\ref{cor:Icoideal} follows that $\Delta$ induces
a linear map
$$
\overline{\Delta}:R_d/I_d(X_1,X_3)\to R_d/I_d(X_2,X_3)\otimes R_d/I_d(X_1,X_2).
$$
Dualizing gives a linear map
$$
\Hom_d(X_1,X_2)\otimes \Hom_d(X_2,X_3)\to \Hom_d(X_1,X_3)
$$
which corresponds to a bilinear multiplication
$$
\Hom_d(X_1,X_2)\times \Hom_d(X_2,X_3)\to \Hom_d(X_1,X_3).
$$
This multiplication is associative, because $R_d^\star$ is associative.
\begin{definition}
For $d\geq \ell(V)$, the category $\category_d(V)$ is the category where the objects
are elements of $\Aff(V)$ and for $X_1,X_2\in V$, $\Hom_d(X_1,X_2)$ is the
set of morphisms from $X_1$ to $X_2$.
\end{definition}
\begin{example}
Consider the group $G=\Gm$ as in Example~\ref{ex:Gm} and \ref{ex:Gm2}.
Let $V=V_{3}\oplus V_5$. We have $\ell(V)=5$.
Let $v_1=(1,1),v_2=(2,1)\in V$. We will compute $\Hom_5(v_1,v_2)$.
The equation $t\cdot v_1=v_2$ gives us the equations
$$
t^3-2,t^5-1
$$
So $S(v_1,v_2)$ is spanned by these two polynomials. 
We have
$$
2t^2-1=(t^5-1)-t^2\cdot (t^3-2)\in (S(v_1,v_2))_5,
$$
$$
t-4=2(t^3-2)-t(2t^2-1)\in ((S(v_1,v_2))_5)_5\subseteq I_5(v_1,v_2)
$$
$$
31=(2t^2-1)-2(t+4)(t-4)\in I(v_1,v_2)_5
$$
Let us assume that $31$ is invertible in $k$. Then we have $1\in I(v_1,v_2)_5$,
so $I(v_1,v_2)_5=R_5$ and $\Hom_5(v_1,v_2)=0$.

Suppose that $X_1=\{v_1\}$ and $X_2=\{(x,x)\mid x\in k\}\subseteq V$.
We will compute $\Hom_5(X_1,X_2)$. The subspace $X_2$ is defined by $x_2-x_1=0$,
and $t\cdot v_1=(t^3,t^5)$, so $S(X_1,X_2)$ is spanned by the polynomial $t^5-t^3$.
We have $t^2-1=t^{-3}(t^5-t^3)\in (S(X_1,X_2))_5$. We have that
$$
S(X_1,X_2))_5=I(X_1,X_2)_5
$$
is the space spanned by
$$
t^3(t^2-1),t^2(t^2-1),\dots,t^{-5}(t^2-1).
$$
The space $R_5/I(X_1,X_2)_5$ is $2$ dimensional and spanned by $1+I_5(X_1,X_2),t+I_5(X_1,X_2)$.
So $\Hom_5(X_1,X_2)$ is $2$-dimensional as well.
\end{example}
\begin{remark}
Let $G=\Ga$ as in Example~\ref{ex:Ga} and \ref{ex:Ga2}.
Suppose that $V$ is a representation with $\ell(V)\leq d$, and $X_1,X_2\in \Aff(V)$.
Then $S(X_1,X_2)$ is spanned by polynomials of degree $\leq d$.
From the Euclidean algorithm in $k[\Ga]\cong k[t]$ follows that $(S(X_1,X_2))\cap R_d=I_d(X_1,X_2)$.
So we have 
$$\Hom_d(X_1,X_2)\neq 0\Leftrightarrow \exists t\in \overline{k}\  t\cdot X_1\subseteq X_2.$$
In particular, $X_1$ and $X_2$ are in the same $G(\overline{k})$ orbit if and only if $X_1\cong_d X_2$.

The same result holds for $G=\Gm$, because $k[\Gm]=k[t,t^{-1}]$ is also an Euclidean domain.
For other groups $G$, it is possible that $X_1\cong_d X_2$ for $X_1,X_2\in \Aff(V)$ with $\ell(V)=d$,
but $X_1,X_2$ are not in the same $G(\overline{k})$-orbit. 
But we know that if $X_1,X_2$ are not in the same $G(\overline{k})$-orbit, then $X_1\not\cong_e X_2$
for some $e\gg 0$.
\end{remark}

\section{Properties of the approximate categories}
\subsection{Some elementary properties}
If $Z$ is a representation with $\ell(Z)\leq d$, then the homomorphism $\mu:Z\to Z\otimes R$ restricts to
$$\mu:Z\to Z\otimes R_d.$$
If $f\in R_d^\star$, then $(\id\otimes f)\circ \mu$ is an endomorphism of $Z$.
One can verify that $Z$ has the structure of an $R_d^\star$-module, where
the multiplication is defined by
$$
f\cdot w:=(\id\otimes f)\circ \mu(w).
$$
For any $X_1,X_2\in \Aff(V)$, $\Hom_d(X_1,X_2)$ is a subspace
of $R_d^\star$, so $\Hom_d(X_1,X_2)$ acts on any representation $Z$ with $\ell(Z)\leq d$.
\begin{lemma}
Suppose that $X_1=v_1+Z_1$ and $X_2=v_2+Z_2$.
If  $f\in \Hom_d(X_1,X_2)$
then $f\cdot X_1\subseteq f(1)v_2+Z_2$.
\end{lemma}
\begin{proof}

For $w\in X_1$ and $h\in Z_2^\perp$ we have
$$
(h\otimes \id)\circ \mu(w)=h(v_2)\otimes 1
$$
If we apply $f$, we have
$$
h(f\cdot w)=h\circ(\id\otimes f)\circ \mu(w)=f\circ (h\otimes \id )\circ \mu(w)=f(h(v_2)\otimes 1)=f(1)h(v_2).
$$
Therefore, $h(f\cdot w-f(1)v_2)=0$, so we get that $f\cdot w\in f(1)v_2+Z_2$.

\end{proof}

%$$
%\mu(Z_1)\subseteq S(X_1,X_2)\otimes W+R_{\ell(W)}\otimes Z_2,
%$$
%so
%$$
%\mu(Z_1)\subseteq I_d(X_1,X_2)\otimes W+R_{d}\otimes Z_2.
%$$
%If we apply $f\otimes \id$,
%we get
%$$
%f(Z_1)=(f\otimes 1)\circ \mu(Z_1)\subseteq f(I_d(X_1,X_2))\otimes W+Z_2=Z_2.
%$$
%\end{proof}
\begin{corollary}If $X_1\cong_d X_2$ then $\dim X_1=\dim X_2$.
\end{corollary}

\begin{lemma}\label{lem:zero_nonzero}
Suppose that $0\in X_1$ and $0\not \in X_2$. Then $\Hom_d(X_1,X_2)=0$.
\end{lemma}
\begin{proof}
Write $X_2=v_2+Z_2$, and choose $f\in Z_2^\perp$ with $f(v_2)\neq 0$.
then 
$$
(f\otimes \id)\circ \mu(0)- f(v_2)\otimes 1= -f(v_2)\otimes 1
$$
is a nonzero multiple of $1\in R_{\ell(V)}$. It follows that $I_d(X_1,X_2)=R_d$ and $\Hom_d(X_1,X_2)=0$.
\end{proof}
\subsection{Constructible equivariants}\label{sec:constructible_equivariants}
\begin{definition}\label{def:equiv}
Suppose that $d$ is a positive integer. We inductively define the notion of a $d$-constructible equivariant:
%For every positive integer $d$ we let $\EE_d$ be the smallest class of equivariants such
%that the following conditions are satisfied:
%Suppose that $G$ is a linearly reductive algebraic group over $k$, and ${\mathcal V}$ is  a set of
%irreducible representations of $G$.
%Let $\Sums({\mathcal V})$ be the set of representations of the form $V_1\oplus \cdots \oplus V_r$,
%with $V_1,\dots,V_r\in {\mathcal V}$
% Let $\Equiv({\mathcal V})$ be the set of all $G$-equivariant polynomial maps
%$\varphi:W\to V$ with $V,W\in \Sums({\mathcal V})$.
%The set of ${\mathcal V}$-constructible equivariants is the smallest subset 
%$\Constr({\mathcal V})\subseteq \Equiv({\mathcal V})$
%such that:
\begin{enumerate}
\item  If $f:V\to V'$ is $G$-equivariant and linear, and $\ell(V),\ell(V')\leq d$, then $f$ is $d$-constructible;
\item if $f_1,f_2:V\to V'$ are $d$-constructible, and $\lambda_1,\lambda_2\in k$, then $\lambda_1f_1+\lambda_2f_2$
is $d$-constructible;
\item if $\ell(V_1)+\ell(V_2)\leq d$, then the bilinear map
$V_1\oplus V_2\to V_1\otimes V_2$ defined by $(v_1,v_2)\mapsto v_1\otimes v_2$ is
$d$-constructible;
%\item  if $f:V\oplus W\to Z$ is $G$-equivariant and bilinear, and $V,W,Z\in \Sums({\mathcal V})$, then
%$f\in \Constr({\mathcal V})$;
%\item if $f_1:V_1\to W_1$ and $f_2:V_2\to W_2$ lie in $\Constr({\mathcal V})$, then $f_1\oplus f_2:V_1\oplus V_2\to %W_1\oplus W_2$ lies in $\Constr({\mathcal V})$;
\item if $f_1:V_1\to V_2$ and $f_2:V_2\to V_3$ are $d$-constructible, then the composition $f_2\circ f_1$ is $d$-constructible.
%\item if $f:V\to W$ and $h:W\to Z$ lie in $\Constr({\mathcal V})$, then $h\circ f:V\to Z$ lies in $\Constr({\mathcal V})$.
\end{enumerate}
We will denote the class of $d$-constructible equivariants by $\Eq_d(V)$.
\end{definition}

\begin{proposition}\label{prop:equiv_separates}
Suppose that $f:V\to V'$ is a $d$-constructible equivariant with $f(v_1)=0$ and $f(v_2)\neq 0$.
Then we have $\Hom_d(v_1,v_2)$. In particular, $v_1$ and $v_2$
are not isomorphic in ${\mathcal C}_d(V)$.
\end{proposition}
The proof will be given after Lemma~\ref{lem:make_functor}
\subsection{$\Eq_{2d}(V)$ is at least as powerful as $\CC_{d}$ }
Let $X$ be a set with $n$ elements.
Consider the symmetric group $G=\Sigma(X)\cong \Sigma_n$.
and let $U\cong k^n$ be the vector space with basis $X$.
The action of $G$ on $U$ gives us a natural inclusion $\tau:G\hookrightarrow \End(U)$. Let $W\subseteq k[G]$
be the vector space spanned by the restrictions of linear and constant functions $\End(U)$ to $G$.
Since $\tau$ is injective, $W$ generates $k[G]$. It is clear that $W$ is stable under the left and right action
of $G$. For $g\in \Sigma_n$ the inverse is just the transpose matrix. From this follows that $\iota(W)\subseteq W$.

We will write $U^{\otimes m}$ for
$$
\underbrace{U\otimes \cdots\otimes U}_m.
$$
To a subset $Y\subseteq X^m$, we can associate a tensor $\tens(Y)\in U^{\otimes m}$ defined by
$$
\tens(Y)=\sum_{(x_1,\dots,x_m)\in Y^m} x_1\otimes \cdots \otimes x_m.
$$

We can define a bilinear  multiplication $\star:U^{\otimes m}\oplus U^{\otimes m}\to U^{\otimes m}$
by 
$$
(x_1\otimes \cdots\otimes x_m)\star (y_1\otimes \cdots \otimes y_m)=
\left\{
\begin{array}{ll}
x_1\otimes \cdots\otimes x_m & \mbox{if $(x_1,\dots,x_m)=(y_1,\dots,y_m)$;}\\
0 & \mbox{otherwise}\end{array}\right.
$$

Define ${\bf 1}=\sum_{x\in X}x$.
For every $i$, we define the linear projection $\pr_i:U^{\otimes d}\to U^{\otimes d}$
by
$$
\pr_i(x_1\otimes \cdots \otimes x_d)=x_1\otimes\cdots\otimes x_{i-1}\otimes {\bf 1}\otimes  x_{i+1}\otimes \cdots \otimes x_m.
$$

%Let ${\bf 1}_m:k\to V^{\otimes m}$ be the equivariant linear map defined by
%$$
%{\bf 1}_m(1)=\sum_{x_1,\dots,x_m\in X}x_1\otimes \cdots\otimes x_m.
%$$

For $m\leq d$, the equivariant maps $\star$ and $\pr_i$ defined above lie in $\EE_{2d}(U^{\otimes d})$.

Suppose that $m_1,\dots,m_s$ are positive integers. Define
$$
V=U^{\otimes m_1}\oplus \cdots \oplus U^{\otimes m_s}\oplus k.
$$
Define ${\bf 1}_d:V\to U^{\otimes d}$ by
$${\bf 1}_d(v_1,\dots,v_s,a)=a ({\bf 1}\otimes \cdots \otimes {\bf 1}).
$$
Then ${\bf 1}_d$ lies in $\EE_{2d}(V)$.

If $Y_i\subseteq X^{m_i}$ for $i=1,2,\dots,s$, then $\Gamma=\langle X,Y_1,\dots,Y_s\rangle$ is a structure
with $s$ relational symbols.
Let $\LL_d=\LL_d(m_1,\dots,m_s)$ be the $d$ variable first order language for this structure, and
let $\CC_d=\CC_d(m_1,\dots,m_s)$ be the $d$-variable  language with counting.
For $\Gamma=\langle X,Y_1,\dots,Y_s\rangle$, define
$$
A_\Gamma:=(\tens(Y_1),\dots,\tens(Y_s),1)\in V.
$$
\begin{definition}
Suppose that $\varphi(x_1,\dots,x_d)$ is a formula in $\CC_d$, and 
$$
f:V\to U^{\otimes d}.
$$
We say that $f$ represents $\varphi$, if
$$
f(A_{\Gamma})=
\sum_{\Gamma\models\varphi(x_1,\dots,x_d)}
x_1\otimes \cdots \otimes x_d.
$$
for all $Y_1,\dots,Y_s$.
\end{definition}
\begin{theorem}
Suppose that  $k$ is a field of characteristic 0 or $p>n$.
Then every formula $\varphi(x_1,\dots,x_d)$  in $\CC_d$
is represented by an equivariant $f\in \EE_{2d}(V)$.
\end{theorem}
\begin{proof}
For $y_1,\dots,y_{m_i}\in \{x_1,\dots,x_d\}$, the  formula $Y_i(x_1,\dots,x_{m_i})$ is represented by an equivariant linear map
$$
V\to U^{\otimes d}.
$$

The formula $x_i=x_j$ is represented by an equivariant linear map.

Suppose that $\varphi_1(x_1,\dots,x_d)$ and $\varphi_2(x_1,\dots,x_d)$
are represented by the covariants $f_1,f_2\in \EE_{2d}(V)$ respectively.
Then $f_1\star f_2$ represents the formula $\varphi_1\wedge \varphi_2$, and $f_1\star f_2\in \EE_{2d}(V)$.

If $\varphi(x_1,\dots,x_d)$ is represented by $f\in \EE_{2d}(V)$, then
$\neg\varphi(x_1,\dots,x_d)$ is represented by ${\bf 1}_d-f$.

Suppose that $q(t)$ is a polynomial in $t$. Define an equivariant $[q(t)]:U^{\otimes d}\to U^{\otimes d}$
by
$$
[q(t)]\big(\sum_{x_1,\dots,x_d\in X} a(x_1,x_2,\dots,x_d)x_1\otimes \cdots \otimes {x_d}\big)=
\sum_{x_1,\dots,x_d\in X}
q(a(x_1,\dots,x_d)) x_1\otimes \cdots \otimes x_d.
$$
If we write $q(t)=tu(t)+a$ then we have
$$
[q(t)](v)=[u(t)](v)\star v+av
$$
It follows by induction on the degree of $q(t)$ that $[q(t)]$ lies in $\EE_{2d}(U^{\otimes d})$ for all polynomials $q(t)$.

Suppose that $\varphi(x_1,\dots,x_d)$ is represented by an equivariant
$f$.  There exists a polynomial $q(t)\in k[t]$ with $q(b)=1$ and $q(i)=0$ for $i\neq b$.
The formula
$$
\exists_b x_i\,\varphi(x_1,\dots,x_d)
$$
is represented by the covariant $[q(t)]\circ\pr_i\circ f$.
\end{proof}
\begin{corollary}\label{cor:orbit_equivar}
Suppose that $k$ is a field of characteristic 0 or $p>n$.
Suppose that $\Gamma=\langle X,Y_1,\dots,Y_s\rangle,\Gamma'=\langle X,Y_1',\dots,Y_s'\rangle$ 
are two structures (for example graphs) and
$$
f(A_\Gamma)=0\Leftrightarrow  f(A_{\Gamma'})=0
$$
for every $f\in \Eq_{2d}(V)$. Then we have
$$
\Gamma\models\varphi\Leftrightarrow \Gamma'\models\varphi'
$$
for every closed formula $\varphi$ in $\CC_d$.
\end{corollary}

\subsection{Constructible functors}\label{sec:constructible}
For the following definition, the reader should bear in mind that $\Hom_d(X_1,X_2)$ is a subspace of $R_d^\star$
for every representation $V$ with $\ell(V)\leq d$ and every two objects $X_1,X_2$ of ${\mathcal C}_d(V)$.
\begin{definition}
We will call a covariant functor ${\mathcal F}:{\mathcal C}_d(V)\to {\mathcal C}_d(V')$
 {\em very faithful} if ${\mathcal F}(\phi)=\phi$ for every morphism.
 A contravariant functor ${\mathcal F}:{\mathcal C}_d(V)\to {\mathcal C}_d(V')$ is called {\em very faithful}
 if ${\mathcal F}(\phi)=\phi\circ\iota$ for every morphism $\phi$.
\end{definition}
Note that a very faithful function ${\mathcal F}:{\mathcal C}_d(V)\to {\mathcal C}_d(V')$ is uniquely
determined by how it acts on objects.
\begin{lemma}\label{lem3.1}
Suppose that $\ell(V),\ell(V')\leq d$. There exist very faithful covariant functors ${\mathcal F},{\mathcal G},{\mathcal H}:{\mathcal C}_d(V)\to {\mathcal C}_d(V')$
such that ${\mathcal F}(X)=\emptyset,{\mathcal G}(X)=\{0\},{\mathcal H}(X)=V'$
for all $X\in \Aff(V)$. Also, for every $\lambda\in k$ there exists a very faithful functor ${\mathcal I}:\category_d(V)\to\category_d(k)$
such that ${\mathcal I}(X)=\{\lambda\}\in \Aff(k)$ for all $X\in \Aff(V)$.
\end{lemma}
\begin{proof}
This is clear because
$$
\Hom_d(\emptyset,\emptyset)=\Hom_d(\{0\},\{0\})=\Hom_d(V_2,V_2)=R_d^\star
$$
and for $\{\lambda\}\in \Aff(k)$, we have $\Hom_d(\{\lambda\},\{\lambda\})=R_d^\star$ as well.
\end{proof}
\begin{lemma}\label{lem3.2}
Suppose that $\ell(V),\ell(V')\leq d$ and $f:V\to V'$ is $G$-equivariant and linear. 
Then there exists a very faithful covariant functor ${\mathcal F}:\category_d(V)\to \category_d(V')$
such that ${\mathcal F}(X)=f(X)$ for all $X\in \Aff(V)$.
\end{lemma}
\begin{proof}
 Suppose that $X_1=v_1+Z_1$, $X_2=v_2+Z_2$ lie in $\category_d(V)$.
Because $f$ is equivariant, we have a commutative diagram
$$
\xymatrix{
V\ar^{\mu}[r]\ar_{f}[d] & V\otimes R_d\ar^{f\otimes \id}[d]\\
V'\ar_{\mu}[r] & V'\otimes R_d}.
$$
We can write
$f(X_1)=f(v_1)+f(Z_1)$ and $f(X_2)=f(v_2)+f(Z_2)$.
The space $S(f(X_1),f(X_2))$ is spanned by elements of the form
$$
(g\otimes \id)\circ \mu(f(w))-g(f(v_2))\otimes 1
$$
where $g\in f(Z_2)^\perp$ and $w\in X_1$. Define $h=g\circ f\in Z_2^\perp$.
We have
\begin{multline*}
(g\otimes \id)\circ \mu(f(w))- g(f(v_2))\otimes 1=
(g\otimes \id)\circ (f\otimes \id)\circ \mu(w)- g(f(v_2))\otimes 1=\\
=
(h\otimes\id)\circ \mu(w)- h(v_2)\otimes 1\in S_d(X_1,X_2).
\end{multline*}
This shows that $S(f(X_1),f(X_2))\subseteq S(X_1,X_2)$.
Following the definitions, it is easy to see that this implies
$\Hom_d(X_1,X_2)\subseteq \Hom_d(\phi(X_1),\phi(X_2))$.
\end{proof}

\begin{lemma}\label{lem3.3}
Suppose that $\ell(V),\ell(V'),\ell(V'')\leq d$, and ${\mathcal F}':{\mathcal C}_d(V)\to {\mathcal C}_d(V')$
and ${\mathcal F}'':{\mathcal C}_d(V)\to {\mathcal C}_d(V'')$ are very faithful covariant (resp. contravariant) functors.
Then there exists a very faithful covariant (resp. contravariant) functor ${\mathcal F}:\category_d(V)\to\category_d(V'\oplus V'')$
such that ${\mathcal F}(X)={\mathcal F}'(X)\oplus {\mathcal F}''(X)$ for all $X\in \Aff(V)$.
\end{lemma}
\begin{proof}
Suppose that $X_1,X_2\in \Aff(V)$,
and let $X_1'={\mathcal F}'(X_1),X_2'={\mathcal F}'(X_2),X_1''={\mathcal F}''(X_1),X_2''={\mathcal F}''(X_2)$.
 It is straightforward to verify that 
$$S(X_1'\oplus X_1'',X_2'\oplus X_2'')=S(X_1',X_2')+S(X_1'',X_2'').
$$
We have
$$
S(X_1',X_2')\subseteq ((S(X_1,X_2)))_d.
$$
Similarly, we have
$$
S(X_1'',X_2'')\subseteq ((S(X_1,X_2)))_d,
$$
so
$$
((S(X_1'\oplus X_1'',X_2'\oplus X_2'')))_d\subseteq ((S(X_1,X_2)))_d.
$$
This implies that 
$$
\Hom_d(X_1,X_2)\subseteq \Hom_d(X_1'\oplus X_1'',X_2'\oplus X_2'')=\Hom_d({\mathcal F}(X_1),{\mathcal F}(X_2)).
$$
\end{proof}

\begin{definition}
Suppose that $X\subseteq V$, $X'\subseteq V'$ are affine subspace.
We define $X\otimes X'\subseteq V\otimes V'$ as the affine subspace
spanned by all $x\otimes x'$ with $x\in X$ and $x'\in X'$.
 Suppose we write $X=v+Z$ and $X'=v'+Z'$
where $v\in V$, $v'\in V'$ and $Z\subseteq X,Z'\subseteq X'$ are subspaces.
 then we have 
 $$X'\otimes X'=v\otimes v'+Z\otimes Z'+kv\otimes Z'+Z\otimes kv'.$$
  \end{definition}

\begin{lemma}\label{lem3.4}
Suppose that $\ell(V),\ell(V')+\ell(V'')\leq d$, and ${\mathcal F}':{\mathcal C}_d(V)\to {\mathcal C}_d(V')$
and ${\mathcal F}'':{\mathcal C}_d(V)\to {\mathcal C}_d(V'')$ are very faithful functors.
Assume that both are covariant (resp. contravariant).
Then there exists a very faithful covariant (respectively contravariant) functor ${\mathcal F}:\category_d(V)\to\category_d(V'\otimes V'')$
such that ${\mathcal F}(X)={\mathcal F}'(X)\otimes{\mathcal F}''(X)$ for all $X\in \Aff(V)$.
\end{lemma}
\begin{proof}
Let $e=\ell(V')$.
 One can verify that
\begin{multline*}
S(X_1'\otimes X_1'',X_2'\otimes X_2'')\subseteq S(X_1',X_2')R_{d-e}+R_{e}S(X_1'',X_2'')\subseteq \\
\subseteq ((S(X_1,X_2)))_eR_{d-e}+R_e((S(X_1,X_2))_{d-e}\subseteq ((S(X_1,X_2))_{d}.
\end{multline*}
It follows that
$$
((S(X_1'\otimes X_1'',X_2'\otimes X_2'')))_{d}\subseteq ((S(X_1,X_2)))_{d}
$$
and therefore
$$
\Hom_{d}(X_1,X_2)\subseteq \Hom_{d}(X_1'\otimes X_1'',X_2'\otimes X_2'')=\Hom_d({\mathcal F}(X_1),{\mathcal F}(X_2))
$$
\end{proof}

\begin{definition}
If $X\subseteq V$ is an affine subspace, then we define
$$
X^+=\{f\in V^\star\mid f(x)=1\mbox{ for all $x\in X$}\}.
$$
\end{definition}
If $0\in X$, then $X^+=\emptyset$. If $0\not\in X$, then $X^{++}=X$.
\begin{lemma}\label{lem3.5}
Suppose that $\ell(V)\leq d$.
There exists a very faithful contravariant functor ${\mathcal D}:{\mathcal C}_d(V)\to {\mathcal C}_d(V^\star)$
such that
$$
{\mathcal D}(X)=X^+
$$
for all $X\in \Aff(V)$.
\end{lemma}
\begin{proof}
Suppose that $X_1,X_2\in \Aff(V)$.
The action of $G$ on $V$ is given by 
$$
\mu:V\to V\otimes R_d.
$$
The action of $G$ on $V^\star$ is given by
$$
\mu^\star:V^\star\to V^\star \otimes R_d
$$
Such that 
$$
(h\otimes \iota)\circ \mu(v)=(v\otimes \id)\circ \mu^\star(h)
$$
for all $h\in V^\star$, $v\in V=V^{\star\star}$.
Suppose that $X_1=v_1+Z_1,X_2=v_2+Z_2\in \Aff(V)$.
The case $0\in X_2$ is clear, because then ${\mathcal D}(X_2)=\emptyset$
and $\Hom_d({\mathcal D}(X_2),{\mathcal D}(X_1))=R_d^\star$.
The case where $0\in X_2$ and $0\not\in X_1$ is also clear, because
we have $\Hom_d(X_1,X_2)=0$.
So we may assume that $0\not\in X_1$ and $0\not\in X_2$.
Choose $u_1,u_2\in V^\star$ with $u_1(X_1)=\{1\}$, and $u_2(X_2)=\{1\}$.
Then we have ${\mathcal D}(X_1)=u_1+Y_1$ and ${\mathcal D}(X_2)=u_2+Y_2$,
where $Y_i=(kv_i+Z_i)^\perp$ for $i=1,2$.
The space $S({\mathcal D}(X_2),{\mathcal D}(X_1))$ is spanned by elements
of the form
\begin{equation}\label{eqfw}
(f\otimes \id)\circ \mu^\star(w)-f(u_1)\otimes 1=(w\otimes \iota)\circ \mu(f)-f(u_1)\otimes 1
\end{equation}
with $f\in Y_1^\perp=kv_1+Z_1$, and $w\in u_2+Y_2\subseteq Z_2^\perp$.
In fact we only need those $f$ for which $f\in X_1=v_1+Z_1$.
Then (\ref{eqfw}) is equal to
\begin{multline*}
(f\otimes \id)\circ \mu^\star(w)-1=(w\otimes \iota)\circ \mu(f)-1=(w\otimes \iota)\circ\mu(f)-w(v_1)=\\=
\iota((w\otimes\id)\circ\mu(f)-w(v_1))
\in \iota(S(X_1,X_2)).
\end{multline*}
From this follows that
$$
\Hom_d(X_1,X_2)\subseteq \iota^\star(\Hom_d({\mathcal D}(X_2),{\mathcal D}(X_1))).
$$
\end{proof}

\begin{definition}\label{def:functor}
We inductively define the notion of a $d$-constructible functor.
\begin{enumerate}
\item The constant functors ${\mathcal F},{\mathcal G},{\mathcal H},{\mathcal I}$ in Lemma~\ref{lem3.1}
and the duality functor ${\mathcal D}$ in Lemma~\ref{lem3.5} are $d$-constructible. The functor ${\mathcal F}$ 
associated to a $G$-equivariant linear map as
in Lemma~\ref{lem3.2} is $d$-constructible.
\item If ${\mathcal F}',{\mathcal F}''$ are as in Lemma~\ref{lem3.3} and they are $d$-constructible, then the very faithful functor ${\mathcal F}$
defined by ${\mathcal F}(X)={\mathcal F}'(X)\oplus {\mathcal F}''(X)$ is $d$-constructible.
\item If ${\mathcal F}',{\mathcal F}''$ are as in Lemma~\ref{lem3.4} and they are $d$-constructible, then the very faithful functor ${\mathcal F}$
defined by ${\mathcal F}(X)={\mathcal F}'(X)\otimes {\mathcal F}''(X)$ is $d$-constructible.
\item A composition of $d$-constructible functors is again $d$-constructible.
\end{enumerate}
\end{definition}

\begin{corollary}\label{cor:dim_equal}
If $X_1,X_2\in \Aff(V)$, $X_1\cong_d X_2$ and ${\mathcal F}:{\mathcal C}_d(V)\to {\mathcal C}_d(V')$
is a $d$-constructible functor. Then ${\mathcal F}(X_1)\cong_d{\mathcal F}(X_2)$.
In particular, we have
$$
\dim {\mathcal F}(X_1)=\dim {\mathcal F}(X_2).
$$
\end{corollary}
\begin{lemma}\label{lem:dconstr}
Suppose that ${\mathcal F},{\mathcal F}':{\mathcal C}_d(V)\to {\mathcal C}_d(V')$ are $d$-constructible functors, 
either both covariant or both contravariant. Then there exists a $d$-constructible functor
${\mathcal G}:{\mathcal C}_d(V)\to {\mathcal C}_d(V')$ defined by ${\mathcal G}(X)={\mathcal F}(X)\cap {\mathcal F}'(X)$.
\end{lemma}
\begin{proof}
If $0\not\in X_1$ and $0\not\in X_2$, then we have $(X_1^+ +X_2^+)^+=X_1\cap X_2$.
In $V\oplus k$, we have
$$
((X_1\times \{1\})^+ +(X_2\times \{1\})^+)^+=X_1\cap X_2\times \{1\}.
$$
So if ${\mathcal I}:{\mathcal C}_d(V)\to {\mathcal C}_d(V\oplus k)$ is just the inclusion,
and ${\mathcal  P}:{\mathcal C}_d(V\oplus k)\to {\mathcal C}_d(V)$ is just the projection,
then we define ${\mathcal G}$ by
$$
{\mathcal G}(X)={\mathcal P}\circ{\mathcal D}({\mathcal D}\circ {\mathcal I}\circ {\mathcal F}(X)+{\mathcal D}\circ{\mathcal I}\circ{\mathcal F}'(X)).
$$
\end{proof}
\begin{lemma}\label{lem:make_functor}
Suppose that $f:V\to V'$ is a $d$-constructible equivariant. Then there exists a $d$-constructible functor
${\mathcal F}:{\mathcal C}_d(V)\to {\mathcal C}_d(V')$ with ${\mathcal F}(\{v\})=\{f(v)\}$ for all $v\in V$.
\end{lemma}
\begin{proof}
This follows easily from the inductive definitions~\ref{def:equiv} and \ref{def:functor}.
\end{proof}
\begin{proof}[Proof of Proposition~\ref{prop:equiv_separates}]
Suppose that $f:V\to V'$ is a $d$-constructible equivariant with $f(v_1)=0$ and $f(v_2)\neq 0$.
By Lemma~\ref{lem:make_functor} there is a $d$-constructible functor ${\mathcal F}:{\mathcal C}_d(V)\to {\mathcal C}_d(V')$ with
${\mathcal F}(\{v\})=\{f(v)\}$ for all $v\in V$. We have $\Hom_d(v_1,v_2)\subseteq \Hom_d(f(v_1),f(v_2))=0$ by Lemma~\ref{lem:zero_nonzero}.
\end{proof}

\begin{lemma}\label{lem:dimF}
Suppose that $v_1,v_2\in V$ and
$$
\dim {\mathcal F}(v_1)=\dim {\mathcal F}(v_2)
$$
for all $d$-constructible functors ${\mathcal F}$. 
Then $f(v_1)=0\Leftrightarrow f(v_2)=0$ for every $d$-constructible equivariant.
\end{lemma}
\begin{proof}
Suppose that $f:V\to V'$ is a $d$-constructible equivariant. There exists a $d$-constructible
functor ${\mathcal F}:{\mathcal C}_d(V)\to {\mathcal C}_d(V')$ with
${\mathcal F}(\{w\})=\{f(w)\}$ for every $w\in V$ by Lemma~\ref{lem:make_functor}.
Define a $d$-constructible functor ${\mathcal F}'$ with ${\mathcal F}'(X)={\mathcal F}(X)\cap \{0\}$. 
Suppose that $w\in V$. If $f(w)=0$, then ${\mathcal F}'(\{w\})=\{0\}$ and $\dim {\mathcal F}'(\{w\})=0$.
If $f(w)\neq 0$, then ${\mathcal F}'(\{w\})=\emptyset$ and $\dim {\mathcal F}'(\{w\})=-\infty$.
\end{proof}

\section{The module isomorphism problem}\label{sec:module_iso}
\subsection{Reformulation of the module isomorphism problem}
 Suppose that $M$ and $N$ are (left) $n$-modules of the 
free associative algebra $T=k\langle x_1,\dots,x_r\rangle$, and we would like to test whether $M$ and $N$
are isomorphic. We can choose a basis in $M$ and identify
$M$ with $k^n$. the action of $x_i$ is given by a matrix $A_i$. Similarly we can identify $N$ with $k^n$.
The action of $x_i$ is given by a matrix $B_i$.
An isomorphism is an invertible linear map $C:M\to N$ such that
$CA_i=B_iC$ for all $i$. This is equivalent to $CA_iC^{-1}=B_i$ for all $i$.
Let $V=\Mat_{n,n}(k)^r$. Then $\GL_n(k)$ acts on $V$ by simultaneous conjugation. The following lemma
follows from the discussion above:
\begin{lemma}
The modules $M$ and $N$ are isomorphic if and only if $A=(A_1,\dots,A_r)$ and $B=(B_1,\dots,B_r)$ lie in the same $\GL_n(k)$-orbit.
\end{lemma}
\begin{proposition}\label{prop:iso}
Let $\overline{k}$ be the algebraic closure of $k$.
The modules $M\otimes_{k}\otimes \overline{k}$ and $N\otimes_k\overline{k}$ are isomorphic if and only if $M\otimes_k\overline{k}$ and $N\otimes_k\overline{k}$ are isomorphic as $T\otimes_k\overline{k}$-modules.
In other words, $A$ and $B$ lie in the same $\GL_n(k)$-orbit if and only if they lie in the
same $\GL_n(\overline{k})$-orbit.
\end{proposition}
\begin{proof}
Suppose that $M\otimes_k\overline{k}$ and $N\otimes_k\overline{k}$ are isomorphic $T\otimes_k\overline{k}$-modules. Then there exists
an invertible matrix $C\in \GL_n(\overline{k})$ with entries in $\overline{k}$ such that $CA_i=B_iC$ for all $i$.
There exists a finite field extension $L$ of $k$ such that all entries of $C$ lie in $L$. It follows from \cite[\S5, Lemma 1]{KR} that $M$ and $N$ are isomorphic
$T$-modules if $k$ is a finite field.
Suppose that $k$ is infinite. 
Choose a basis $h_1,h_2,\dots,h_r$ of $L$ as a $k$-vector space. We can write
$$
C=\sum_j h_jC_j
$$
where $C_j$ is an $R$-module homomorphism from $M$ to $N$ for all $j$. Let $C(s_1,s_2,\dots,s_r)=\sum_{j=1}^r s_jC_j$
where $s_1,\dots,s_r$ are indeterminates. Since $C(h_1,\dots,h_r)$ is invertible, we have
$\det C(h_1,\dots,h_r)\neq 0$.
So $\det C(s_1,\dots,s_r)$ is not
the zero polynomial. Since $k$ is infinite, we can choose $a_1,\dots,a_r\in k$ such that $\det C(a_1,\dots,a_r)\neq 0$.
Then $C(a_1,\dots,a_r)$ is an isomorphism between $M$ and $N$.
\end{proof}

\begin{theorem}[See~\cite{CIK,BL}]\label{theo:module_polytime}
There exists an algorithm for determining whether two $n$-dimensional modules $M$ and $N$ are isomorphic
which requires only a polynomial number (polynomial in $n$)  of arithmetic operations in the field $k$.
\end{theorem}
%An efficient algorithms for determining whether two modules are isomorphic were given in \cite{CIK, BL}.
%These algorithms only needs a polynomial number of arithmetic operations in the field $k$.
%Consider the category $\Rmod$ of finite dimensional left $R$-modules. Then $M$ and $N$ are isomorphic
%if and only if $\Hom_R(M,N)$ contains an invertible element.
%The space $\Hom_R(M,N)$ can easily be computed:
%$$
%\Hom_R(M,N)=\{T\in \Hom_k(M,N)\mid \forall i\,TA_i=B_iT\}.
%$$
%Heuristically, if $k$ is infinite and 
If $k$ is a fixed finite field, then this algorithm runs in polynomial time. 
Even if $k$ is not fixed, if $k=\F_q$ and $\log q$ grows polynomially, then the algorithm still runs in polynomial time.

\subsection{The isomorphism problem in $k$-categories}\label{sec:cat_iso}
A category ${\mathcal C}$ is a $k$-category, if $\Hom_{\mathcal C}(M,N)$ is a vector space
for every two objects $M$ and $N$, and the composition map
$$
\Hom_{\category}(M,N)\times \Hom_{\category}(N,P)\to \Hom_{\category}(M,P)
$$
is $k$-bilinear for all objects $M,N,P$.
Assume we have any $k$-category $\category$ with finite dimensional $\Hom$-spaces.
Suppose that $M$ and $N$ are two isomorphic objects in $\category$, and let
$T=\Hom_\category(N,N)$. Then $T$ and  $\Hom_\category(M,N)$ are isomorphic as left $T$-modules.
\begin{lemma}\label{lem:varphi_iso}
Suppose that $M$ and $N$ are isomorphic, and
$\psi:T\to \Hom_\category(M,N)$ is an isomorphism of $T$-modules.
Then $\varphi=\psi(1)$ is an isomorphism between $M$ and $N$.
\end{lemma}
\begin{proof}
Since $1$ generates $T$ as an $T$-module, $\varphi=\psi(1)$ generates $\Hom_\category(M,N)$ as an $T$-module.
Suppose that $\gamma:M\to N$ is an isomorphism. Since $\gamma\in T\varphi$, there exists $\tau\in \Hom_\category(N,N)=T$
such that $\gamma=\tau\varphi$. So $\varphi$ has a left inverse.
The map
$$
\Phi:\Hom_\category(N,M)\to \Hom_{\category}(N,N).
$$
defined by $\Phi(\lambda)=\varphi\lambda$.
is injective because $\varphi$ has a left inverse.
Since $\dim \Hom_\category(N,M)=\dim \Hom_\category(N,N)<\infty$
we have that $\Phi$ is surjective. Therefore $\id_N$ lies in the image of $\Phi$.
This implies that $\varphi$ has a right inverse as well.
\end{proof}

%If $\Hom_\category(M,N)=R \varphi$, then $\varphi$ is an isomorphism.
To test whether any two objects $M$, $N$ are isomorphic, we can proceed as follows.
\begin{enumerate}
\item  First test whether $T$ and $\Hom_\category(M,N)$ are isomorphic as $T$-modules. 
If they are not isomorphic, then $M$ and $N$ are not isomorphic. Otherwise let $\psi:T\to \Hom_\category(M,N)$
be an isomorphism of $R$-modules.
\item Let $\varphi=\psi(1)$.
Test whether $\varphi$ is an isomorphism. This is easy, because testing whether $\varphi$ has a left and a right inverse
just boils down to a system of linear equations.
Now $M$ and $N$ are isomorphic if and only if $\varphi$ is an isomorphism.
\end{enumerate}

We can use this approach for the categories ${\mathcal C}_d(V)$. Note that
$$
\dim\Hom_d(X_1,X_2)\leq \dim R_d
$$
for all $d$ because $\Hom_d(X_1,X_2)$ is a subspace of $R_d^\star$.
\begin{proof}[Proof of Theorem~\ref{theo:polytime}]
We have reduced the isomorphism problem in ${\mathcal C}_d(V)$ to the isomorphism problem of modules, 
and by Theorem~\ref{theo:module_polytime} the isomorphism problem of modules can be solved in a polynomial number
of arithmetic operations in the field $k$.
\end{proof}
 
Let $\overline{k}$ be the algebraic closure of $k$. We construct a new category ${\mathcal C}\otimes_k \overline{k}$,
where the objects are the same as the objects of ${\mathcal C}$, but
$$
\Hom_{{\mathcal C}\otimes_k \overline{k}}(M,N)=\Hom_{\mathcal C}(M,N)\otimes_k\overline{k}.
$$
\begin{proposition}\label{prop:kbar}
Suppose that $M,N$ are objects in ${\mathcal C}$. If $M\otimes_k\overline{k},N\otimes_k\overline{k}$ are isomorphic in ${\mathcal C}\otimes_k\overline{k}$,
then they are isomorphic in ${\mathcal C}$.
\end{proposition}
\begin{proof}
Suppose $M$ and $N$ are objects in ${\mathcal C}$ which are isomorphic in ${\mathcal C}\otimes_k\overline{k}$.
Let $T=\Hom_{\category}(N,N)$. Then $T\otimes_k\overline{k}$ is isomorphic to 
$\Hom_{\mathcal C}(M,N)\otimes_k\overline{k}$ as a $T\otimes_k\overline{k}$-module.
From Proposition~\ref{prop:iso} follows that $T$ and $\Hom_{\category}(M,N)$ are isomorphic as $T$-modules.
Let $\psi:T\to \Hom_{\mathcal C}(M,N)$ be an isomorphism and define $\varphi=\psi(1)$.
Then $\psi$ extends to an isomorphism $\psi\otimes \id:T\otimes_k\overline{k}:T\otimes_k\overline{k}\to
\Hom_{\mathcal C}(M,N)\otimes_k\overline{k}$ of $T\otimes_k\overline{k}$-modules.
and $\varphi\otimes 1=\psi(1)$ is an isomorphism by Lemma~\ref{lem:varphi_iso}. We can write
$$
(\varphi\otimes 1)^{-1}=\sum_{i=1}^l \gamma_i\otimes a_i
$$
where $a_1,a_2,\dots,a_l\in \overline{k}$ are linearly independent over $k$ and $a_1=1$.
Then we have
$$
\id=(\varphi\otimes 1)\circ (\varphi\otimes 1)^{-1}=\sum_{i=1}^l (\varphi\gamma_i)\otimes a_i.
$$
It follows that $\varphi\gamma_i=\id$ for $i=1$ and $\varphi\gamma_i=0$ for $i>1$.
Therefore, $\varphi$ has a right inverse. Similarly $\varphi$ has a left inverse, so $\varphi$ is an isomorphism.
\end{proof}
\begin{proof}[Proof of Theorem~\ref{theo:general}]
The implication (i)$\Rightarrow$(ii) follows from Corollary~\ref{cor:dim_equal}.
It is easy to verify that the category ${\mathcal C}_d(V\otimes_k\overline{k})$ (working over the field $\overline{k}$)
is equal to ${\mathcal C}_d(V)\otimes_k\overline{k}$. 
Suppose that $X_1,X_2\in \Aff(V)$ are in the same $G(\overline{k})$-orbit, say
$g\cdot X_1=X_2$ for some $g\in G(\overline{k})$. 
We may view $g$ as an element of $R_d^*\otimes_k\overline{k}$ if we identify $g$ with
the function $R_d\otimes_k\overline{k}\to \overline{k}$ which is evaluation at $g$.
Then $g\in \Hom_d(X_1,X_2)\otimes_k\overline{k}$,
and $g^{-1}\in \Hom_d(X_2,X_1)\otimes_k\overline{k}$ is its inverse.
This shows that $X_1,X_2$ are isomorphic in ${\mathcal C}_d(V)\otimes_k\overline{k}$. By Proposition~\ref{prop:kbar},
we have that $X_1,X_2$ are isomorphic in ${\mathcal C}_d(V)$.
The implication (ii)$\Rightarrow$(iii) follows.
\end{proof}
\begin{proof}[Proof of Theorem~\ref{theo:equis}]
The implication (i)$\Rightarrow$(ii) follows from Lemma~\ref{lem:dimF}. The other implications
follow from Theorem~\ref{theo:general}.
\end{proof}
\begin{proof}[Proof of Theorem~\ref{theo:graph_implications}]
The implication (i)$\Rightarrow$(ii) is obvious because ${\mathscr C}_d$ contains ${\mathscr L}_d$.
The equivalence (ii)$\Leftrightarrow$(iii) is Theorem~\ref{theo:CFI}.
The implication (ii)$\Rightarrow$(iv) follows from Corollary~\ref{cor:orbit_equivar}. 
The implications (iv)$\Rightarrow$(v)$\Rightarrow$(vi)$\Rightarrow$(vii) follow from Theorem~\ref{theo:equis}.
The equivalence (vii)$\Leftrightarrow$(viii) is Lemma~\ref{lem:graph_orbit}.
\end{proof}

\subsection{The categories ${\mathcal C}_d(V)$ for the general linear group}
Let $G$ be the group $ \GL_n(k)$. Let $U=k^n$ be the standard $n$-dimensional representation.
We can identify $G$ with the variety 
$$
\{(C,D)\in \Hom(k^n,U)\times \Hom(U,k^n)\mid DC=I_n\}\subseteq U^n\times (U^\star)^n.
$$
Let $W$ be the subspace of $k[G]$ spanned by the constant functions,
and the functions induced by linear functions on $U^n\times (U^\star)^n$. 
So $W$ is isomorphic to $U^n\oplus (U^\star)^n\oplus k$ as a representation of $G$.
We have $\ell(U)=\ell(U^\star)=1$.
This  choice of $W$ gives us now a filtration of $R=k[G]$.
For an $n$-dimensional vector space $V$ every weakly decreasing sequence $\lambda=(\lambda_1,\dots,\lambda_n)\in \Z^n$ 
corresponds to an irreducible representation $S^\lambda(V)$ of $\GL(V)$.
If $\lambda_r>0$ and $\lambda_{r+1}\leq 0$ for some $r$, then we have that $S^\lambda(V)$
is a subrepresentation of
$$
U^{\otimes (\lambda_1+\cdots+\lambda_r)}\otimes (U^\star)^{\otimes (-\lambda_{r+1}-\cdots-\lambda_n)}.
$$
It follows that  $\ell(S^\lambda(V))\leq \sum_{i=1}^n |\lambda_i|$, where $|\cdot|$ denotes the absolute value.

Define
$$
V=\Mat_{n,n}(k)^r=\End(U)^r
$$
where $G$ acts on $V$ by simultaneous conjugation. We have $\ell(V)=2$. 

The remainder of the Section is dedicated to the prove of Proposition~\ref{prop:not_converse}.
Let $T=k\langle x_1,\dots,x_r\rangle$ be the free associative algebra with $r$ generators, and $M$  and $N$ be an $n$-dimensional
$R$-modules representated by $A=(A_1,\dots,A_r)\in V$ and
$B=(B_1,\dots,B_r)\in V$ respectively.

Let $\mbox{$T$--mod}$ be the category of finite dimensional left $R$-modules.
\begin{proposition}
There exists a functor ${\mathcal F}:{\mathcal C}_3(V)\to\mbox{\rm $T$--mod}$ such
that for every $n$-dimensional module $T$-module $M$ that is represented by $A=(A_1,\dots,A_r)$
we have ${\mathcal F}(A)\cong M$.
\end{proposition}
\begin{proof}
Let $(C,D)\in G$.
We can write $C=(c_{i,j})$ and $D=(d_{i,j})$. Then
$I_3(A,B)$ is the $3$-truncated ideal generated by the entries of the matrices
$CD-I$, $DC-I$ and $CA_iD-B_i$ for $i=1,2,\dots,r$.
Then the entries of $CA_i-B_iC=(CA_iD-B_i)C-CA_i(DC-I)$ also lie in $I_3(A,B)$.
The coordinate functions $C=(c_{i,j})$ define a linear map $\pi:R^\star \to \Hom_k(k^n,k^n)$.
It follows that $\pi(\Hom_d(M,N))\subseteq \Hom_R(M,N)$.
So we define ${\mathcal F}(\phi)=\pi(\phi)$ for all $\phi\in \Hom_d(M,N)$.
\end{proof}
\begin{corollary}
The elements $A=(A_1,\dots,A_r)$, $B=(B_1,\dots,B_r)$ lie in the same orbit if and only if $A$ and $B$ are isomorphic to
${\mathcal C}_3(V)$.
\end{corollary}
%\subsection{Constructible functors are more expressive than constructible equivariants}
%Suppose that $G$ is an algebraic group and $V$ is a representation.
%The group $G$ acts on the coordinate ring $k[V]$ of $V$, and $k[V]^G$ denotes the ring
%of invariant polynomials. An invariant $f\in k[V]^G$ is constant on orbits. So
%if $v,w\in V$ and $f(v)\neq f(w)$ then $v$ and $w$ do not lie in the same orbit.
%If $G$ is not finite, then invariants may not separate distinct orbits. For example,
%if $w$ lies in the Zariski closure of the orbit $G\cdot v$, then $f(v)=f(w)$ for every 
%polynomial invariant.
%
%If $G$ is finite, then invariants can be used to separate any two distinct orbits. However,
%the number of invariants needed may be excessive large. Or, the size of the invariants,
%written explicitly as a polynomial, may be very large. For an application of invariant theory
%to the graph isomorphism problem, see \cite{PT,Thiery}.
%
%A $G$-equivariant polynomial map $\varphi:V\to W$, where $V$ and $W$ are representations of $G$
%we will call an {\em equivariant}. Orbits can be separated by equivariants: If $v,w\in V$
%then there exists an equivariant $\varphi$ with either $\varphi(v)=0$ and $\varphi(w)\neq 0$,
%or, $\varphi(w)=0$ and $\varphi(v)\neq 0$.
%
%Suppose that $G$ is linearly reductive. An equivariant $\varphi:V\to W$ where $W$ is irreducible is called
%a {\em covariant}. In this case, orbits can be separated by covariants. 
Consider now the case were $r=1$.
As the following proposition shows, the size needed for a covariant to distinguish two orbits may
be excessively large:
\begin{proposition}\label{prop:counterexample}
Let 
$$
C=\begin{pmatrix}
0 & 1\\
0 & 0
\end{pmatrix},D=\begin{pmatrix}
0 & 0\\
0 & 0
\end{pmatrix}\in \Mat_{2,2}(\C).
$$
and define the block matrices
$$
A=\begin{pmatrix}
C & & & & \\
& C &&&\\
& & \ddots &&\\
& & & C &\\
&&&& C
\end{pmatrix}
\mbox{ and }
B=\begin{pmatrix}
C & & & & \\
& C &&&\\
& & \ddots &&\\
& & & C &\\
&&&& D
\end{pmatrix}
$$
%Suppose that 
%$$
%A=\begin{pmatrix}
%0 & 1 &&&&&&&&\\
%0 & 0 &&&&&&&&\\
%&& 0 & 1 &&&&&&\\
%&& 0 & 0 &&&&&&\\
%&&&&\ddots &&&&&\\
%&&&&&& 0 & 1&&\\
%&&&&&& 0 & 0 &&\\
%&&&&&&&& 0 & 1\\
%&&&&&&&& 0 & 0
%\end{pmatrix}
%\mbox{ and }
%B=\begin{pmatrix}
%0 & 1 &&&&&&&&\\
%0 & 0 &&&&&&&&\\
%&& 0 & 1 &&&&&&\\
%&& 0 & 0 &&&&&&\\
%&&&&\ddots &&&&&\\
%&&&&&& 0 & 1&&\\
%&&&&&& 0 & 0 &&\\
%&&&&&&&& 0 & 0\\
%&&&&&&&& 0 & 0
%\end{pmatrix}
%$$
 in $\Mat_{2n,2n}(\C)$. The group $\GL_{2n}(\C)$ acts on $\Mat_{2n,2n}(\C)$ by conjugation.
Then $A,B$ do not lie in the same $\GL_{2n}(\C)$ orbit.
If $\varphi:V\to V'$ is a covariant which distinguishes the orbits of $A$ and $B$ respectively,
then we have $\dim(V')\geq 3^n$ and $\ell(V')\geq 2n$.
\end{proposition}
\begin{proof}
Define $V=\End(U)$. Then $V\cong \Mat_{2n,2n}(\C)$, and $\GL(U)\cong \GL_{2n}(\C)$.
We can write $U=U_1\oplus \cdots\oplus U_n$ where $U_i\cong \C^2$.
We can view $\End(U_1)\oplus \cdots \oplus \End(U_n)$ as a subalgebra of $\End(U)$.
Now $A,B\in \End(U_1)\oplus \cdots \oplus \End(U_n)\subseteq \End(U)$ are given by
$A=(C,C,\dots,C)$ and $B=(C,C,\dots,C,D)$.

Suppose that $\varphi:V\to V'$ is a covariant, where $V'$ is an irreducible representation of $\GL(V)$.
If $\varphi(A)=0$, then $\varphi(B)=0$ because $B$ lies in the orbit closure of $A$.
Suppose that $\varphi(A)\neq 0$ and $\varphi(B)=0$.
As a representation of $\GL(U_1)\times \cdots \times \GL(U_n)$, $V'$ may not be irreducible.
Let $Z_1\otimes \cdots \otimes Z_n$ be an irreducible summand of $V'$
as a $\GL(U_1)\times \cdots \times \GL(U_n)$ representation,
such that $p(\varphi(A))\neq 0$, where $p$ is the $\GL(U_1)\times \cdots \times \GL(U_n)$-equivariant projection
 $V'\to Z_1\otimes \cdots \otimes Z_n$, and $Z_i$ is an irreducible representation of $\GL(U_i)$ for all $i$.
Let $q:\End(U_1)\oplus \cdots \oplus \End(U_n)\to Z_1\otimes \cdots\otimes Z_n$ be the restriction of $p\circ \varphi$.
We have $q(A)\neq 0$.
Suppose that $\dim Z_i=1$ for some $i$. 
Let $B'=(C,\dots,C,D,C,\dots,C)$. Then $B$ and $B'$ are in the same $\GL(V)$-orbit,
so $\varphi(B')=0$, and hence $q(B')=0$. Now $B'$ lies in the $\SL(U_i)$-closure of of $A$.
Since $q$ is $\SL(U_i)$-invariant, we get $q(A)=q(B')=0$.
Contradiction.
Hence $\dim Z_i\geq 2$. Since $Z_i$ must be an irreducible representation of $\PSL(U_i)$,
 we even have $\dim Z_i\geq 3$.
It follows that
$$
\dim V'\geq (\dim Z_1)\cdots (\dim Z_n)\geq 3^n.
$$
Let us write $V'=S^\lambda(U)$ for some $\lambda=(\lambda_1,\dots,\lambda_{2n})\in \Z^{2n}$
and $Z_i=S^{\mu^{(i)}}(U_i)$, where
$\mu^{(i)}=(\mu^{(i)}_1,\mu^{(i)}_2)$.

 From $\dim(Z_i)\geq 3$ follows that $\mu_1^{(i)}-\mu_2^{(i)}\geq 2$
for all $i$.
 The representation of $\GL(U)$ restricts to
$\GL(U_1)\times \cdots\times \GL(U_n)$ according to the Littlewood-Richardson rule.
We have the following inequalities:
$$
\lambda_1\geq \sum_{i}\mu_1^{(i)}
$$
and
$$
\lambda_{2n}\leq \sum_i\mu_2^{(i)}
$$
Taking the difference gives us
$$
|\lambda|=\sum_{i} |\lambda_i|\geq \lambda_1-\lambda_{2n}\geq \sum_{i=1}^n \mu_1^{(i)}-\mu_2^{(i)}\geq 2n.
$$
It follows that $\ell(V')\geq 2n$.
\end{proof}
\begin{remark}
Define $\varphi:\End(U)\to \End(\textstyle \bigwedge^n U)$ by
$$
\varphi(E)=E\wedge\cdots \wedge E.
$$
Then $\varphi(A)\neq 0$ and $\varphi(B)=0$. Note that $\End(\textstyle \bigwedge^n U)$ is not irreducible.
There exists an irreducible summand $W$ of $\End(\bigwedge^n U)$ such that $p(\varphi(C))\neq 0$,
where $p:\End(\bigwedge^n U)\to W$ is the projection. If we set $q=p\circ\varphi$,
then $q$ is a covariant that distinguishes the orbits of $A$ and $B$.
Note that $ \dim W\leq \dim \End(\bigwedge^n U)\leq 4^n$.
\end{remark}

\section{The Cai-F\"urer-Immerman examples}\label{sec:CFI}
Cai, F\"urer and Immerman showed that for every postive integer $d$
there exist non-isomorphic $2$-colored graphs $\Gamma$ and $\Gamma'$ such that
$\Gamma\sim_d\Gamma'$. To explain this result, we need to describe the construction
of Cai, F\"urer and Immerman which, given a graph $Q$, two nonisomorphic
2-colored graphs $\Gamma(Q)$ and $\Gamma'(Q)$ (see \cite[\S6]{CFI}).

Suppose that $Q=\langle X,R\rangle$ is a graph.  Let $E=\{\{x,y\}\mid (x,y)\in R\}$ be the set of edges in the graph.
For every vertex $x\in X$, we define $E(x)\subseteq E$ by $E(x)=\{e\in E\mid x\in e\}$.
So $E(x)$ is the set of edges which are incident with $x$.
We define a vertex set $X(Q)=X_1(Q)\cup X_2(Q)$, where
$$
X_1(Q)=\{c_{x,Y}\mid x\in X, Y\subseteq E(x),\mbox{ $|Y|$ is even}\},
$$
and
$$
X_2(Q)=\{a_{x,e}\mid x\in X,e\in E(x)\}\cup
\{b_{x,e}\mid x\in X,e\in E(x)\}.
$$
We define the edge set $E(Q)$ by
\begin{multline*}
E(Q)=\{\{a_{x,e},c_{x,Y}\}\mid x\in X,e\in Y\}\cup
\{\{b_{x,e},c_{x,Y}\}\mid x\in X, e\not\in Y\}\cup\\
\cup\{\{a_{x,e},a_{a,e}\}\mid x,y\in X,e\in E(x)\cap E(y)\}\cup
\{\{b_{x,e},b_{a,e}\}\mid x,y\in X,e\in E(x)\cap E(y)\}
\end{multline*}
We also define another edge set $E'(Q)$ as follows: 
We choose two special vertices $\widetilde{x},\widetilde{y}$ such that
$\widetilde{e}=\{\widetilde{x},\widetilde{y}\}\in E$ is an edge.
To obtain $E'(Q)$, 
remove $\{a_{\widetilde{x},\widetilde{e}},a_{\widetilde{y},\widetilde{e}}\}$ and $\{b_{\widetilde{x},\widetilde{e}},b_{\widetilde{y},\widetilde{e}}\}$ from $E(Q)$ and add
$\{a_{\widetilde{x},\widetilde{e}},b_{\widetilde{y},\widetilde{e}}\}$ and $\{a_{\widetilde{y},\widetilde{e}},b_{\widetilde{x},\widetilde{e}}\}$.
Let $R(Q)$ and $R'(Q)$ be the symmetric relations corresponding to the edge sets $E(Q)$ and $E(Q')$ respectively.
We now have two 2-colored graphs: $\Gamma(Q)=(X(Q),R(Q),X_1(Q),X_2(Q))$
and $\Gamma'(Q)=(X(Q),R'(Q),X_1(Q),X_2(Q))$.

The following proposition follows from Lemma 6.2 of \cite{CFI}. We will give a proof here, because a crucial lemma
is based on this proof.
\begin{proposition}\label{prop:not_isomorphic}
The graphs $\Gamma(Q)$ and $\Gamma'(Q)$ are not isomorphic.
\end{proposition}
\begin{proof}
Let $M$ be the adjacency matrix of $\Gamma(Q)$ with entries in the field $\F_2$. 
Since $X(Q)=X_1(Q)\cup X_2(Q)$,
$M$ has the following block form:
$$
M=
\begin{pmatrix}
 A_{1,1} & A_{1,2}\\
A_{2,1} & A_{2,2}
\end{pmatrix}
$$
where $A_{1,1}$, $A_{2,2}$  are symmetric and $A_{1,2}=A_{2,1}^t$. 
Similarly, let 
$$
M'=\begin{pmatrix}
 A_{1,1} & A_{1,2}\\
A_{2,1} & A_{2,2}'
\end{pmatrix}
$$
be the adjacency matrix for $\Gamma'(Q)$.

Let 
$$
B=\begin{pmatrix}
   A_{2,1} & A_{2,2}
  \end{pmatrix},
\mbox{ and }
B'=\begin{pmatrix}
    A_{2,1} & A_{2,2}'
   \end{pmatrix}
$$
The proposition now follows from the lemma below.
\end{proof}
\begin{lemma}\label{lem:not_same_rank}
 We have
$$
\rank(B)=3|E|+|X|-2
$$
and
$$
\rank(B')=3|E|+|X|-1.
$$
\end{lemma}
\begin{proof}
The image of $\im(B)$ of $B$ is equal
to $\im(A_{2,1})+\im(A_{2,2})$. The space $\im(A_{2,1})$ is spanned by all
$$
\sum_{e\in Y} a_{x,e}+\sum_{e\in E(x)\setminus Y} b_{x,e}
$$
with $x\in X$, $Y\subseteq E(x)$ with $|Y|$ even, and $\im(A_{2,2})$ is spanned by
all
$$
a_{x,e}+a_{y,e},b_{x,e}+b_{y,e}
$$
with $x\in X$ and $e\in E(x)$. It is clear that $\dim\im(A_{2,2})=2|E|$.
For $e=\{x,y\}\in E$, define $a_e=a_{x,e}+\im(A_{2,2})=a_{y,e}+\im(A_{2,2})$
and $b_e=b_{x,e}+\im(A_{2,2})=b_{y,e}+\im(A_{2,2})\in k^{4|E|}/\im(A_{2,2})$.
Now $\im(B)/\im(A_{2,2})$ is spanned by all 
$$
\sum_{e\in Y}a_{e}+\sum_{e\in E(x)\setminus Y}b_e
$$
where $x\in X$ and $Y\subseteq X$ with $|Y|$ even.
Note that $a_{e}+b_{e}+a_{f}+b_f\in \im(B)/\im(A_{2,2})$ for all $x\in X$, $e,f\in E(x)$.
Since $Q$ is connected, it follows that $a_e+b_e+a_f+b_f\in \im(B)/\im(A_{2,2})$
for all $e,f\in E$. Let $Z\subseteq \im(B)$ containing $\im(A_{2,2})$
such that $Z/\im(A_{2,2})$ is spanned by all $a_e+b_e+a_f+b_f$.
The dimension of $Z/\im(A_{2,2})$ is $|E|-1$.
Now $\im(B)/Z$ is spanned by all elements of the form
$$
\sum_{e\in E(x)}b_e+Z
$$
with $x\in X$. Since $Q$ is connected, it follows that  $\dim \im(B)/Z=|X|-1$.
We conclude that
$$
\rank(B)=\dim\im(B)=2|E|+(|E|-1)+(|X|-1)=3|E|+|X|-2.
$$

We can do a similar computation for $\rank(B')$. First of all $\dim\im(A_{2,2}')=2|E|$.
Let $\widetilde{e}=\{\widetilde{x},\widetilde{y}\}$ be the special edge.
For $e=\{x,y\}\neq\widetilde{e}$, we define $a'_e=a_{x,e}+\im(A_{2,2}')$
and $b'_e=b_{x,e}+\im(A_{2,2}')$. For $\widetilde{e}=\{\widetilde{x},\widetilde{y}\}$
we define $a_{\widetilde{e}}'=a_{\widetilde{x},\widetilde{e}}+\im(A_{2,2}')=b_{\widetilde{y},\widetilde{e}}+\im(A_{2,2}')$
and $b_{\widetilde{e}}'=b_{\widetilde{x},\widetilde{e}}+\im(A_{2,2}')=a_{\widetilde{y},\widetilde{e}}\in\im(A_{2,2}')$.
Let $Z'\subseteq B$ be he space containing $\im(A_{2,2}')$
such that $Z'/\im(A_{2,2}')$ is spanned by all $a_e'+b_e'+a_f'+b_f'$ with $e,f\in E$.
We have $\dim (Z'/\im(A_{2,2}'))=|E|-1$.
Finally, $\im(B')/\dim(Z')$ is spanned by all 
$$
\sum_{e\in E(x)} b_{e}'+Z'
$$
with $x\in X$ and $x\neq \widetilde{y}$, and
$$
\big(\sum_{e\in E(\widetilde{x})\setminus\{\widetilde{e}\}}a_{e}'\big)+b_{\widetilde{e}'}.
$$
It is easy to see that $\dim(\im(B')/\dim(Z'))=|X|$. So we obtain
$$
\dim(\im(B'))=2|E|+(|E|-1)+|X|=3|E|+|X|-1.
$$

\end{proof}
\begin{definition}
A separator of a graph $Q=(X,R)$ is a subset $Y\subset X$ such that the induced subgraph
on $X\setminus Y$ has no connected component with more than $|X|/2$ vertices.
\end{definition}
The following theorem is Theorem 6.4 in \cite{CFI}.
\begin{theorem}
Suppose that $Q$ is a graph such that every separator of $Q$ has at least $d+1$ vertices. 
Then $\Gamma(Q)$ and $\Gamma'(Q)$ cannot be distinguished by the $d$-variable logic with counting.
\end{theorem}
There exists a family of graphs $T_d$ with the following properties: $T_d$ has $O(d)$ vertices, ever vertex in $T_d$ has degree $3$,
and  every separator has at least $d+1$ vertices. Then $\Gamma(T_d)$ and $\Gamma'(T_d)$ have $O(d)$ vertices,  and
$\Gamma(T_d)\sim_d \Gamma(T_d')$. Every vertex of $\Gamma(T_d)$ or $\Gamma'(T_d)$ has degree 3.
This shows that for fixed $d$,  the $d$-dimensional Weisfeiler-Lehman algorithm cannot distinguish all graphs of degree $3$.
However, it is possible to distinguish graphs of bounded degree in polynomial time. Such an algorithm was given in \cite{Luks}.

\begin{proof}[Proof of Theorem~\ref{theo:CFInoniso}]
Suppose that $k=\F_2$ and
  $Q=\langle X,E\rangle$. We will show that $A_{\Gamma(Q)}$ and $A_{\Gamma'(Q)}$ can be separated by
a $3$-constructible functor.
%Let $\Gamma(Q)=\langle X(Q),R(Q),X_1(Q),X_2(Q)\rangle$
%and $\Gamma'(Q)=\langle X(Q),R(Q)',X_1(Q),X_2(Q)\rangle$,
%whee $B_\Gamma$ and $B_{\Gamma}'$ are the adjacency matrices of $\Gamma(Q)$ and $\Gamma(Q)'$).
We have $A_{\Gamma},A_{\Gamma_1}\in V=U\otimes U\oplus U\oplus U\oplus k$. Let $p_1,p_2:V\to U$
be the two projections onto $U$, and $q:V\to U\otimes U\cong \End(U)$ be the projection onto $U\otimes U$.
Then we have $p_2(A_{\Gamma(Q)})=p_2(A_{\Gamma'(Q)})=\sum_{x\in X_2(Q)} x$.
Let $\delta:U\to U\otimes U$ defined by $\delta(x)=x\otimes x$ for all $x\in X(Q)$.
Then $\delta(p_2(A_{\Gamma(Q)}))=\delta(p_2(A_{\Gamma'(Q)}))\in U\otimes U\cong U\otimes U^\star\cong \End(U)$
is the projection of onto the span of $X_2(Q)$.
The compositions $q(A_{\Gamma(Q)})\circ \delta(p_2(A_{\Gamma(Q)}))$ and
$q(A_{\Gamma'(Q)})\circ \delta(p_2(A_{\Gamma'(Q)})$ are given by the matrices $B$ and $B'$ in the proof of Proposition~\ref{prop:not_isomorphic}.
Define the following $3$-constructible functors:
The functor
$$
{\mathcal F}_1:{\mathcal C}_3(V)\to {\mathcal C}_3(\End(U))
$$
is defined by the $3$-constructible equivariant linear map $\delta\circ p_2$.
The functor
$$
{\mathcal F}_2:{\mathcal C}_3(\End(U))\to {\mathcal C}_3(\End(U)\otimes U)$$
is defined by
$$
{\mathcal F}_2(Z)=Z\otimes U.
$$
The functor
$$
{\mathcal F}_3:{\mathcal C}_3(\End(U)\otimes U)\to{\mathcal C}_3(U)
$$
is defined by the equivariant $f\otimes v\mapsto f(v)$.
Let ${\mathcal F}_4:{\mathcal C}_3(V) \to {\mathcal C}_3(\End(U))$
defined by the equivariant linear map $q$.
Then ${\mathcal F}_3\circ {\mathcal F}_2\circ {\mathcal F}_1$ and ${\mathcal F}_4$ are $3$-constructible,
and
$$
{\mathcal F}_4\otimes ({\mathcal F}_3\circ {\mathcal F}_2\circ {\mathcal F}_1):{\mathcal C}_3(V)\to {\mathcal C}_3(\End(U)\otimes U).
$$
is constructible.
Define a $3$-constructible functor ${\mathcal G}:{\mathcal C}_3(V)\to {\mathcal C}_3(U)$ by
$$
{\mathcal G}={\mathcal F}_3\circ ({\mathcal F}_4\otimes ({\mathcal F}_3\circ {\mathcal F}_2\circ {\mathcal F}_1)).
$$
Then we have ${\mathcal G}(A_{\Gamma(Q)})=\im B$ and ${\mathcal G}(A_{\Gamma'(Q)})=\im B'$.
By Lemma~\ref{lem:not_same_rank}, we have
$\dim {\mathcal G}(A_{\Gamma(Q)})\neq \dim {\mathcal G}(A_{\Gamma'(Q)})$,
so ${\mathcal G}$ distinguishes $A_{\Gamma(Q)}$ and $A_{\Gamma'(Q)}$.
\end{proof} 
\section{Open problems}
We finish with some open questions:
 
\begin{problem}
Does ${\bf AC}_d$ distinguish all pairs of non-isomorphic graphs for some $d$?
\end{problem}
A positive answer to this problem implies that the Graph Isomorphism Problem lies in the complexity class {\bf P}.

Suppose that $\Gamma_1,\Gamma_2$ are (colored) graphs constructed using the Cai-F\"urer-Immerman method. We know that $A_{\Gamma_1}$
and $A_{\Gamma_2}$ are non-isomorphic in ${\mathcal C}_3(V)$, assuming we are working over the field $\F_2$ (see Theorem~\ref{theo:CFInoniso}). The proof
heavily relies on the fact that we are working over the field $\F_2$. So a natural question to ask is:
\begin{problem}
Are $A_{\Gamma_1}$ and $A_{\Gamma_2}$ non-isomorphic in
 ${\mathcal C}_3(V)$, even if we are  working over a field of characteristic other than 2?
\end{problem}
If we work over a base field $k=\Q$, then the size of the rational numbers may grow exponentially if we do
arithmetic operations such as multiplications and additions. So it is a priori not clear that algorithms for testing isomorphism in ${\mathcal C}_d(V)$
run in polynomial time. 
\begin{problem}
If we work over the basefield $k=\Q$, can we test for isomorphism in ${\mathcal C}_d(V)$ in polynomial time?
\end{problem}
One may expect that there is a probabilistic algorithm for testing isomorphism in ${\mathcal C}_d(V)$ by working over $\F_p$
for various random primes $p$ for which $\log(p)$ is polynomial in the number of vertices.


\begin{thebibliography}{200}
 \bibitem{BGM}L.~Babai, D.~Yu.~Grigoryev, D.~M.~Mount, {\it Isomorphism of graphs with bounded eigenvalue multiplicity},
 Proceedings of the 14th Annual ACM Symposium on Theory of Computing, 1982, 310--324.
 \bibitem{BS} R.~G.~Busacker, T.~L.~Saaty, {\it Finite Graphs and networks}, McGraw-Hill, New York, 1965.
 \bibitem{BL}P.~A.~Brooksbank, E.~M.~Luks, {\it Testing isomorphism of modules}, J. Algebra~{\bf 320} (2008), no.~11, 4020--4029. 

 
\bibitem{CFI}J.-Y. Cai, M.~F\"urer, N.~Immerman, {\it An optimal lower bound for the number of variables
for graph identification}, Combinatorica~{\bf 12} (1992).
\bibitem{CIK}A. Chistov, G.~Ivanyos, M.~Karpinski, {\it Polynomial-time algorithms for modules over finite dimensional algebras}, in: Preceedings Int. Symp. on Symbolic and Algebraic Computation (ISSAC), 1997, 68--74.
  \bibitem{EKP}S. Evdokimov, M.~Karpinski, I. Ponomarenko, {\it On a new higher dimensional Weisfeiler-Lehman algorithm},
Journal of Algebraic Combinatorics~{\bf 10} (1999).
\bibitem{EP}S.~Evdokimov, I.~Ponomarenko, {\it On highly closed cellular algebras and highly closed isomorphisms},
Electronic J. of Comb.~{\bf 6} (1999), \#18.
\bibitem{FM}I.~S.~Filotti, J.~N.~Mayer, {\it A polynomial-time  algorithm for determining the isomorphism of graphs of fixed genus},
Proceedings of the 12th Annual ACM Symposium on Theory of Computing, 1980, 236--243.
\bibitem{Fr}S.~Friedland, {\it Coherent algebras and the graph isomorphism problem},  Discr. Appl. Math.~{\bf 25} (1989), 73-98.
\bibitem{GL}J.~J.~Graham, G.~I.~Lehrer, {\it Cellular algebras}, Invent. Math.~{\bf 123} (1996), no.~1, 1--34.
\bibitem{Higman2} D.~G.~Higman, {\it Coherent configurations. I},  Rend. Sem. Mat. Univ. Padova~{\bf  44} (1970), 1--25. 
\bibitem{Higman}D.~G.~Higman, {\it Coherent algebras}, Linear Algebra Appl.~{\bf 93} (1987), 209--240.
\bibitem{HT}J.~E.~Hopcraft, R.~E.~Tarjan, {\it A $V\log V$ algorithm for isomorphism of triconnected planar graphs}, 
J. Comput. System Sci.~{\bf 7} (1973), 323--331. 
\bibitem{HW}J.~E.~Hopcroft, J.~K.~Wong, {\it Linear time algorithm for isomorphism of planar graphs}, 
Proceedings of the Sixth Annual ACM Symposium on Theory of Computing, Assoc. Comput. Mach., 1974, 172--184.
\bibitem{KUJ}J.~K\"obler, U.~Sch\"oning, J.~Tor\'an, {\it The graph isomorphism problem: its structural complexity},
Progress in Theoretical Computer Science, Birkh\"auser, Boston, 1993.
\bibitem{KR}H.~Kraft, Ch.~Riedtmann, {\it Geometry of representations of quivers}, in: {\it Representations of Algebras: proceedings of the Durham Symposium, 1985},
P. Webb (ed.),  London Mathematical Society Lecture Note Series~{\bf 116}, Cambridge University Press, 1986.
\bibitem{Luks}E.~M.~Luks, {\it
Isomorphism of graphs of bounded valence can be tested in polynomial time}, 
J. Comput. System Sci.~{\bf 25} (1982), no.~1, 42--65. 
\bibitem{Luks2}E.~M.~Luks, {\it Parallel algorithms for permutation groups and graph isomorphism}, Proc. IEEE Symp. Foundations of Computer Science, , 1986, 292--302.
\bibitem{Miller}G.~Miller, {\it Isomorphism testing for graphs of bounded genus}, Proceedings of the 12th Annual ACM Symposium on Theory of Computing, 1980, 225--235.
\bibitem{MS} K.~Mulmuley, M.~Sohoni, {\it Geometric Complexity Theory I: an approach to the P vs. NP and related problems}, SIAM J. Comput.~{\bf 31} (2001), 496--526.
\bibitem{MS2}K.~Mulmuley, M.~Sohoni, {\it Geometric Complexity Theory II: towards explicit obstructions for embeddings among class varieties}, SIAM J.~Comput.~{\bf 38} (2008).
\bibitem{PZ}M.~Pouzet, N.~Thi\'ery, {\it
Invariants alg\'ebriques de graphes et reconstruction},
C. R. Acad. Sci. Paris S\'er. I Math. ~{\bf 333} (2001), no.~9, 821--826. 
\bibitem{Rama}S.~Ramanujan, {\it A proof of Bertrand's postulate}, J. Indian Math. Soc.~{\bf 11} (1919), 181--182. Also in:  Collected papers of Srinivasa Ramanujan,   208--209, AMS Chelsea Publ., Providence, RI, 2000. 

\bibitem{Thiery}N.~M.~Thi\'ery, {\it  Computing minimal generating sets of invariant rings of permutation groups with SAGBI-Gr\"obner basis}, Discrete models: combinatorics, computation, and geometry (Paris, 2001), 315--328 (electronic), Discrete Math. Theor. Comput. Sci. Proc., AA, Maison Inform. Math. Discr\`et. (MIMD), Paris, 2001.
\bibitem{V}L.~.G.~Valiant, {\it The complexity of the permanent}, Theoret. Comp. Sci.~{\bf 8} (1979), 189--201.
\bibitem{Weis}B.~Yu.~Weisfeiler (ed.), {\it On Construction and Identification of Graphs},  Lecture Notes in Mathematics~{\bf 558}, 
Springer, Berlin-New York, 1976.
\bibitem{WL}B.~Yu.~Weisfeiler, A.~A.~Lehman, {\it A Reduction of a Graph to a Canonical Form and 
an Algebra arising dusing this Reduction} (in Russian),
Nauchno-Tekhnicheskaya  Informatsia, Seriya 2~{\bf 9} (1968), 12--16.
\bibitem{W}V.~Weispfenning, {\it 
Some bounds for the construction of Gr\"obner bases},  Applicable algebra, error-correcting codes, combinatorics and computer algebra (Karlsruhe, 1986), 195--201, 
Lecture Notes in Comput. Sci.~{\bf 307}, Springer, Berlin, 1988. 


 \end{thebibliography}
\end{document}